\documentclass[a4paper, 11pts]{amsart}
\title[{\tiny perverse coherent sheaves on blow-ups}]
{perverse coherent sheaves on blow-ups at codimension 2 loci}
\date{}
\author{Naoki Koseki}

\makeatletter
 
  \@addtoreset{equation}{section}
 \makeatother

\usepackage{amsmath, amssymb, amsthm, amscd}
\usepackage[frame,cmtip,curve,arrow,matrix,line,graph]{xy}

\theoremstyle{plain}
\newtheorem{thm}{Theorem}[section]
\newtheorem{prop}[thm]{Proposition}
\newtheorem{lem}[thm]{Lemma}
\newtheorem{cor}[thm]{Corollary}
\newtheorem*{thm*}{Theorem}

\theoremstyle{definition}
\newtheorem{defin}[thm]{Definition}

\newtheorem*{NaC}{Notation and Convention}
\newtheorem*{ACK}{Acknowledgement}

\theoremstyle{remark}
\newtheorem{rmk}[thm]{Remark}

\newtheorem{ques}[thm]{Question}
\newtheorem{ex}[thm]{Example}

\newtheorem{claim}[thm]{Claim}

\DeclareMathOperator{\ch}{ch}
\DeclareMathOperator{\td}{td}

\DeclareMathOperator{\Per}{Per}
\DeclareMathOperator{\Spec}{Spec}
\DeclareMathOperator{\id}{id}
\newcommand{\dR}{\mathbf{R}}
\newcommand{\dL}{\mathbf{L}}
\DeclareMathOperator{\Hom}{Hom}
\DeclareMathOperator{\Coh}{Coh}
\DeclareMathOperator{\Supp}{Supp}
\DeclareMathOperator{\Ker}{Ker}
\DeclareMathOperator{\Coker}{Coker}
 
\DeclareMathOperator{\NS}{NS}

\DeclareMathOperator{\ext}{ext}
\DeclareMathOperator{\Ext}{Ext}
\DeclareMathOperator{\Image}{Image}
\DeclareMathOperator{\Hilb}{Hilb}
\DeclareMathOperator{\Sym}{Sym}

\DeclareMathOperator{\Bl}{Bl}
\DeclareMathOperator{\PQuot}{P-Quot}
\DeclareMathOperator{\PSub}{P-Sub}
\DeclareMathOperator{\Quot}{Quot}
\DeclareMathOperator{\Sub}{Sub}
\DeclareMathOperator{\Cone}{Cone}
\DeclareMathOperator{\Exc}{Exc}
\DeclareMathOperator{\pt}{pt}

\begin{document}

\begin{abstract}
Let $f \colon X \to Y$ be 
the blow-up of a smooth projective 
variety $Y$ along its codimension two 
smooth closed subvariety. 
In this paper, we show that 
the moduli space of stable sheaves on 
$X$ and $Y$ are connected by a 
sequence of flip-like diagrams. 
The result is a higher dimensional 
generalization of the result of 
Nakajima and Yoshioka, which is 
the case of $\dim Y=2$. 
As an application of our general result, 
we study the birational geometry of 
the Hilbert scheme of two points. 
\end{abstract}

\maketitle

\section{Introduction}
\subsection{Main result}
In this paper, we consider the following 
natural question: 
\begin{ques}
\label{question}
Let $X \dashrightarrow Y$ be a 
birational map between varieties. 
What is the relation between 
the moduli space of stable sheaves on $Y$ 
and that of $X$ ?
\end{ques}

There are several works 
answering Question \ref{question} 
in various situations 
(\cite{bb18, bri02, mw97, ny11a, ny11b, ny11c, tod13b, tod14}). 
In particular, 
Nakajima and Yoshioka proved the 
following theorem: 

\begin{thm}[\cite{ny11b}]
\label{ny}
Let $(Y, H)$ be a polarized smooth projective surface, 
$f \colon X \to Y$ the blow-up at a point. 
Let $v=(v_{0}, v_{1}, v_{2}) \in H^{2*}(X; \mathbb{Q})$ be 
the Chern character of a sheaf 
with $v_{0} >0$, 
$\gcd(v_{0}, f^*H. v_{1})=1$. 
 
Then there exists a diagram of projective schemes 
\begin{equation}
\label{zig-zag}
\xymatrix{
&\cdots & M^m(v) \ar[ld] \ar[rd]^{\xi_{m}^{-}} & 
&M^{m+1}(v) \ar[ld]_{\xi_{m}^+} \ar[rd] & &\cdots \\
& & &M^{m, m+1}(v) & & & 
}
\end{equation}
such that 
\begin{enumerate}
\item[(1)] For an integer 
$m \in \mathbb{Z}_{\geq 0}$, 
the scheme 
$M^m(v)$ is the moduli space of 
{\it $m$-stable sheaves} 
with Chern character $v$, 
and $M^{m, m+1}(v)$ is a scheme 
whose closed points corresponds to 
$m$-stable and $(m+1)$-stable sheaves with various 
Chern characters
(see Definition \ref{m-stability} 
for the notion of $m$-stability). 

\item[(2)] If there exists an element 
$w \in H^{2*}(Y; \mathbb{Q})$ with $v=f^*w$, 
then $M^0(v)$ is isomorphic to 
the moduli space of stable sheaves on $Y$. 

\item[(3)] For every sufficiently large $m$, 
$M^m(v)$ is isomorphic to 
the moduli space of stable sheaves on $X$. 

\item[(4)] For every integer 
$m \in \mathbb{Z}_{\geq 0}$, 
the fiber over $\xi_{m}^{\pm}$ 
is the Grassmann variety. 
\end{enumerate}
\end{thm}

In the above theorem, the notion of 
{\it perverse coherent sheaves} 
plays an important role. 
In particular, it leads us to define 
natural stability conditions indexed by 
$m \in \mathbb{Z}_{\geq 0}$, called 
{\it $m$-stablilty}. 
The $m$-stability is 
similar to the Gieseker stability. 
However, an $m$-stable sheaf may have 
a torsion subsheaf supported on 
the $f$-exceptional curve. 
Such a generalization of stability enables us 
to connect the moduli space of stable sheaves 
on $X$ and that of $Y$. 

The notion of perverse coherent sheaves 
was introduced by Bridgeland (cf. \cite{bri02}). 
A perverse coherent sheaf is 
an element of the heart 
of a certain bounded t-structure 
(called {\it perverse heart} and 
denoted by $\Per(X/Y)$) 
in the derived category of 
coherent sheaves on $X$. 
The heart $\Per(X/Y)$ encodes 
the information of the morphism $f$ 
and it can be defined more general situation. 
In particular, if we have the blow-up 
$f \colon X \to Y$ 
of a smooth projective variety along 
its codimension two smooth closed subvariety, 
we can define 
the perverse heart 
$\Per(X/Y) \subset D^b(X)$. 
In this setting, we generalize 
the result of the paper \cite{ny11b}. 
The precise statement of our main theorem 
of the present paper is the following: 

\begin{thm}
\label{main thm}
Let $(Y, H)$ be a polarized smooth projective variety 
of dimension $n \geq 2$, 
$f \colon X \to Y$ the blow-up 
along its codimension two smooth closed subvariety. 
Let $v=(v_{0}, v_{1}, \cdots) \in H^*(X; \mathbb{Q})$ be 
a Chern character with $v_{0} >0$, 
$\gcd(v_{0}, f^*H^{n-1}. v_{1})=1$. 
Then there exists a diagram of projective schemes 
as in (\ref{zig-zag}) 
such that 
\begin{itemize}
\item[(1)] The scheme $M^m(v)$ is the moduli space of 
$m$-stable sheaves with Chern character $v$, 
and $M^{m, m+1}(v)$ is a scheme whose 
closed points corresponds to 
$m$-stable and $(m+1)$-stable sheaves 
with various Chern characters 
($m \in \mathbb{Z}_{\geq 0}$). 

\item[(2)] If there exists an element 
$w \in H^{2*}(Y; \mathbb{Q})$ with $v=f^*w$, 
then $M^0(v)$ is isomorphic to 
the moduli space of stable sheaves on $Y$. 

\item[(3)'] For every sufficiently large $m$, 
the moduli space of stable sheaves on $X$ 
is embedded into $M^m(v)$ 
as an open and closed subscheme. 

\item[(4)'] For every integer 
$m \in \mathbb{Z}_{\geq 0}$, 
the fiber over $\xi_{m}^{\pm}$ 
is the certain Quot scheme. 
\end{itemize}
\end{thm}

The above theorem is a summary of 
Corollary \ref{xi is isom}, Proposition \ref{moduli on X}, 
Proposition \ref{scheme}, and Proposition \ref{fiber}. 
See the next subsection 
for the summary of differences between 
Theorem \ref{ny} and Theorem \ref{main thm}.


Using Theorem \ref{main thm}, 
we also study the birational geometry 
of Hilbert scheme of two points: 

\begin{thm}[Corollary \ref{flip}, Collollary \ref{extremal contraction}]
\label{application}
Let $f \colon X \to Y$ be 
as in Theorem \ref{main thm}, 
let $v:=(1, 0, \cdots, 0, -2) \in H^{2*}(X; \mathbb{Q})$. 
Assume that $H^1(\mathcal{O}_{Y})=0$. 
Then we have a diagram of projective varieties 
\[
\xymatrix{
& &\ar[ld]^{\xi_{0}} \widetilde{M}^1(v) \ar[rd]_{\xi_{1}^{-}} & &\Hilb^2(X) \ar[ld]^{\xi_{1}^{+}} \subset M^2(v) \\
&\Hilb^2(Y) & &M^{1, 2}(v) & 
}
\]
such that 

\begin{enumerate}
\item When $\dim Y=2$ 
(resp. $3$, $\geq 4$), 
$\Hilb^2(X) \dashrightarrow \widetilde{M}^1(v)$ 
is a flip (resp. a flop, an anti-flip). 

\item The morphism $\xi_{0}$ is 
the contraction of a $K$-negative 
extremal ray. 
\end{enumerate}
\end{thm}

Moreover, we will determine 
all the fibers over $\xi_{0}$. 
When $\dim Y \geq 3$, we see that 
some fibers of $\xi_{0}$ are {\it not} 
the Grassmann varieties (see Lemma \ref{determine fiber}). 
This is the new phenomenon 
which does not happen in dimension $2$ 
(see Theorem \ref{ny} (4)).


\subsection{Difference between Theorem \ref{ny} and Theorem \ref{main thm}}
The idea of the proof of our main theorem 
is similar to that of \cite{ny11b}. 
However, we need to modify the 
proofs in various points, 
which are not so straightforward. 
Let us explain about the differences. 
Let $f \colon X \to Y$ 
be as in Theorem \ref{main thm}, 
$D \subset X$ the $f$-exceptional divisor. 

{\bf $0$-stability and $1$-stability}. 
One of the key part of the argument is 
to describe the difference between 
$0$-stability and $1$-stability. 
Here, a coherent sheaf $E$ is said to be $0$-stable 
if $E$ is a perverse coherent sheaf 
and $f_{*}E$ is a slope stable torsion free sheaf on $Y$. 
On the other hand, we say that 
$E$ is $1$-stable if 
$E(-D)$ is $0$-stable. 
To explain the argument, 
let $E^{-} \in \Coh(X)$ 
be a $0$-stable sheaf. 
In the surface cace, 
Nakajima and Yoshioka proved that 
the obstruction for the $1$-stability 
is captured by looking at the 
vector space 
$V:=\Hom\left(
\mathcal{O}_{D}(-1), E^-
\right)$. 
In fact, we can show that 
the evaluation morphism 
$ev \colon V \otimes \mathcal{O}_{D}(-1) \to E$ 
is always injective and hence we have 
a short exact sequence 
\begin{equation}
\label{evaluation}
0 \to V \otimes \mathcal{O}_{D}(-1) \to E^- \to F:=\Coker(ev) \to 0
\end{equation}
in $\Coh(X)$. 
Furthermore, we can also show that 
$F$ is $1$-stable. 
Similarly, we can construct 
a $0$-stable sheaf from 
a $1$-stable sheaf $E^+$ 
by looking at the vector space 
$\Hom\left(
E^+, \mathcal{O}_{D}(-1)
\right)$. 
In this way, we get a 
(set-theoretical) diagram (\ref{zig-zag}) 
by sending $E^- \in M^0(v)$ 
to $F \in M^{0, 1}(v)$, etc. 

However, in higher dimension, 
not only the sheaf 
$\mathcal{O}_{D}(D)$, 
but also various subsheaves of $E^-$ 
become the obstruction for $1$-stability. 
Hence we should take the 
{\it maximum} subsheaf among them. 
We take such a subsheaf by using {\it torsion pairs} 
on $\Coh(X)$. 
See Definition \ref{torsion pair} 
for the definition of torsion pairs. 
Assume that we have a torsion pair on $\Coh(X)$. 
Then we have the canonical decomposition 
\[
0 \to T \to E^- \to F \to 0
\]
with respect to the torsion pair. 
Using such a decomposition, 
we will see that the difference 
between $0$-stability and 
$1$-stability is captured 
by certain torsion pairs on 
$\Coh(X)$, which are defined in 
Definition \ref{torsion pair for per} 
and Definition \ref{torsion pair for pertilde}. 
In the surface case, we can easily see that 
the exact sequence (\ref{evaluation}) 
is nothing but the decomposition with respect to 
our torsion pair. 

{\bf Scheme structure on $M^{0, 1}(v)$.} 
The another key point is how to define 
the scheme structure on $M^{0, 1}(v)$. 
As a set, $M^{0, 1}(v)$ is the disjoint union of 
the moduli spaces of 
$0$-stable and $1$-stable sheaves 
with various Chern characters. 
In the surface case, 
Nakajima and Yoshioka used 
the moduli space of 
{\it perverse coherent systems} 
to define the scheme structure on $M^{0, 1}(v)$. 

Instead, we use the natural morphisms 
between moduli spaces. 
More precisely, we will show that 
the set-theoretical diagram 
is naturally identified with 
the diagram 
\[
\xymatrix{
&M^0(v) \ar[rd]_{\xi=f_{*}} & &M^1(v) \ar[ld]^{\xi^{+}=f_{*}} \\
& &M^{H}(v'), & 
}
\]
where $M^H(v')$ denotes the moduli space of 
Gieseker stable sheaves on $Y$. 
Hence we define a scheme structure on 
$M^{0, 1}(v)$ as the union of the 
scheme-theoretic images of $\xi$ and $\xi^+$. 
To be more precise, let 
$I, I^+ \subset \mathcal{O}_{M^H(v')}$ 
be ideal sheaves defining the scheme theoretic images 
of $\xi, \xi^+$, respectively.  
Then the scheme $M^{0, 1}(v)$ is defined to be 
a closed subscheme of $M^H(v')$ whose defining ideal 
is $I \cap I^+$. 
Note that 
Nakajima and Yoshioka do the 
essentially same thing 
and our approach is inspired by them. 

{\bf Moduli space of stable sheaves on $X$.}
We mention about the difference between 
Theorem \ref{ny} (3) and 
Theorem \ref{main thm} (3)'. 
According to (3)' in Theorem \ref{main thm}, 
we have an open and closed embedding 
from the moduli space of stable sheaves on $X$ 
to the moduli space of $m$-stable sheaves 
with sufficiently large integer 
$m \in \mathbb{Z}_{\geq 0}$. 
In the surface case, Nakajima and Yoshioka 
showed that the above embedding is actually an isomorphism. 
To show that, they used the following speciality 
of the surface: 
Let $E$ be a torsion free sheaf on a surface. 
Then the quotient $E^{DD}/E$ is $0$-dimensional, 
where $E^{DD}$ is the double dual of $E$. 
In particular, we have 
$\chi(E^{DD}/E) \geq 0$. 
In the higher dimension case, 
we do not have such a positivity of the 
quotient sheaf $E^{DD}/E$, 
which is crucial in the proof given by 
Nakajima and Yoshioka. 
In the present paper, we only prove that 
when $n=3$ and $v=(1, 0, \cdots, 0, -k)$, 
the embedding given in Theorem \ref{main thm} (3)' 
is actually an isomorphism (see Example \ref{ex in hilb}).


\subsection{Plan of the paper}
The paper is organized as follows. 
In Section \ref{perverse sheaf}, 
we collect the notions and the properties 
about perverse coherent sheaves on blow-ups. 
In Section \ref{wall crossing}, 
we describe the diagram (\ref{zig-zag}) 
set-theoretically. 
In the proofs, we will use certain torsion pairs. 
In Section \ref{Gieseker}, 
we explain the relationship between 
the $m$-stability and the Gieseker stability 
on both blow-up and blow-down varieties. 
In Section \ref{scheme str}, 
we realize the diagram (\ref{zig-zag}) 
scheme-theoretically. 
In Section \ref{Hilb}, 
we study the diagram (\ref{zig-zag}) 
more explicitly in the case of 
Hilbert scheme of two points.


\begin{NaC}
In this paper, we always work over 
the complex number field $\mathbb{C}$. 
We use the following notanions: 
\begin{itemize}
\item For a variety $X$, 
we denote by $\Coh(X)$ 
the Abelian category of 
coherent sheaves on $X$. 

\item For a variety $X$, 
we denote by 
$D^b(X):=D^b(\Coh(X))$ 
the bounded derived category of 
coherent sheaves on $X$. 

\item For a set of objects 
$S \subset D^b(X)$, 
we denote by $\langle S \rangle$ 
the extension closure of $S$ 
in $D^b(X)$. 

\item For a proper morphism 
$f \colon M \to N$ 
between varieties 
and $E, F \in \Coh(M)$, 
we denote by 
$\mathcal{E}xt^q_{f}(E, F)$ 
the $q$-th derived functor of 
$f_{*}\mathcal{H}om(E, F)$. 

\item For a variety $X$ and 
$E, F \in \Coh(X)$, we define 
$\hom(E, F):=\dim\Hom(E, F)$ 
and 
$\ext^i(E, F):=\dim\Ext^i(E, F)$. 

\end{itemize}
\end{NaC}


\section{Perverse coherent sheaves and their moduli spaces}
\label{perverse sheaf}
Throughout the paper, 
we use the following notations: 
Let $Y$ be a smooth projective variety, 
$C \subset Y$ a codimension $2$ 
smooth closed subvariety of $Y$. 
Let $X:=\Bl_{C}Y$ be the blow-up of $Y$ along $C$, 
$D \subset X$ the exceptional divisor. 
Hence we have the following diagram: 

\[
\xymatrix{
&D \ar[d]_{\pi} \ar@{^{(}->}[r]^-j &X \ar[d]^f \\
&C \ar@{^{(}->}[r]_{i} &Y. 
}
\]


\subsection{Perverse coherent sheaves on blow-ups}
In this subsection, we collect the results 
about the perverse coherent sheaves. 
First note that the following two conditions hold: 
\begin{enumerate}
\item $\dR f_{*}\mathcal{O}_{X} \cong \mathcal{O}_{Y}$, 
\item $\dim f^{-1}(y) \leq 1$ for all $y \in Y$. 
\end{enumerate}

In such a situation, 
Bridgeland introduced 
the heart of a bounded t-structure  
$\Per(X/Y)$ on $D^b(X)$ 
(called {\it perverse heart})
as follows (cf. \cite{bri02, vdb04}): 

\[
\Per(X/Y):=\left\{ E \in D^b(X): 
\begin{array}{lll}
f_{*}\mathcal{H}^{-1}(E)=0, \\
\dR^1f_{*}\mathcal{H}^0(E)=0, 
\Hom(\mathcal{H}^0(E), \mathcal{C}^0)=0, \\
\mathcal{H}^i(E)=0, \quad i \neq -1, 0
\end{array}
\right\}, 
\]
where 
$
\mathcal{C}:=
\left\{
E \in D^b(X) : 
\dR f_{*}E=0
\right\} 
$
and 
$\mathcal{C}^0:=\mathcal{C} \cap \Coh(X)$. 
In \cite{bri02}, the heart $\Per(X/Y)$ is denoted 
as $^{-1}\Per(X/Y)$. 
We call an element of 
$\Per(X/Y) \cap \Coh(X)$ as a 
{\it perverse coherent sheaf}. 

We will use the following lemma: 

\begin{lem}
\label{calc}
We have equalities 
\[
\mathcal{C}^0=\pi^*\Coh(C) \otimes \mathcal{O}(D)
=f^*\Coh(C) \otimes \mathcal{O}(D). 
\]
\end{lem}

\begin{proof}
The equality 
$\pi^*\Coh(C) \otimes \mathcal{O}(D)
=f^*\Coh(C) \otimes \mathcal{O}(D)$ 
follows from the isomorphism of functors 
\[
f^*i_{*} \cong j_{*}\pi^* \colon 
\Coh(C) \to \Coh(X). 
\]

To prove the first equality, 
first note that we have 
$\mathcal{C}=\pi^{*}D^b(C) \otimes \mathcal{O}(D)$ 
by the following semi-orthogonal decomposition
(cf. \cite{orl92}) 
\[
\label{stsod}
D^b(X)=
\left\langle
\pi^{*}D^b(C) \otimes \mathcal{O}(D), 
\dL f^{*}D^b(Y) 
\right\rangle. 
\]

Furthermore, by the exactness of the functors 
$\pi^*$ and $(-) \otimes \mathcal{O}(D)$, 
we have the inclusion 
\[
\mathcal{C}^0 \subset \pi^*\Coh(C) \otimes \mathcal{O}(D). 
\]
Since both of them are hearts of bounded t-structures 
on $\mathcal{C}$, they must coincide. 
\end{proof}

The following result is due to Van den Bergh: 

\begin{thm}[{\cite[Proposition 3.3.2]{vdb04}}]
\label{vdb}
Let $\mathcal{E}:=\mathcal{O}_{X} \oplus \mathcal{O}_{X}(-D)$. 
Then we have an equivalence of triangulated categories 

\[
\Phi := \dR f_{*}\dR\mathcal{H}om(\mathcal{E}, *) : 
D^b(X) \xrightarrow{\cong} D^b(\Coh(\mathcal{A})), 
\]
where $\mathcal{A}:= f_{*}\mathcal{E}nd(\mathcal{E})$. 
Furthermore, the functor $\Phi$ 
restricts to an equivalence 
$\Per(X/Y) \cong \Coh(\mathcal{A})$ 
of Abelian categories. 
\end{thm}

From the above theorem, we can easily get the following: 
\begin{lem}
\label{r1 per}
Let $E \in \Per(X/Y)$. 
Then we have 
$\dR^1 f_{*}\left(
E(D)
\right)=0$. 
\end{lem}
\begin{proof}
For $E \in \Per(X/Y)$, we have 
$
\Phi(E)=\dR f_{*}E \oplus \dR f_{*}\left(E(D) \right) 
\in \Coh(\mathcal{A})$ 
by Theorem \ref{vdb}. 
This in particular implies 
$\dR^1f_{*}(E(D))=0$. 
\end{proof}

We give a criterion when 
a coherent sheaf $E \in \Coh(X)$ is in $\Per(X/Y)$. 
Before stating the criterion, we recall the following lemma: 

\begin{lem}[{\cite[Lemma 1.2]{ny11b}}]
\label{property of adjoint}
Let $E$ be a coherent sheaf, 
$\phi: f^{*}f_{*}E \to E$ be 
the adjoint morphism. 
Then the following statements hold. 
\begin{enumerate}
\item We have 
$f_{*}(\Image\phi) \cong f_{*}E$, and 
$\dR^1f_{*}(\Image\phi)=0$. 

\item We have 
$f_{*}(\Coker\phi)=0$, and 
$\dR^1f_{*}(\Coker\phi) \cong \dR^1f_{*}E$. 

\item $E \in \Per(X/Y)$ if and only if 
$\Coker\phi=0$. 

\item We have $\Ker(\phi) \in \mathcal{C}^0$. 
\end{enumerate}
\end{lem}

The following criterion will be 
frequently used in this paper: 

\begin{lem}[cf. {\cite[Proposition 1.9]{ny11b}}]
\label{criterion for per}
Let $E$ be a coherent sheaf. 
Then $E$ is an object of the category 
$\Per(X/Y)$ if and only if 
for every point $y$ of $C$, we have 
$\Hom(E, \mathcal{O}_{L_{y}}(-1))=0$, 
where 
$L_{y}:=f^{-1}(y) \cong \mathbb{P}^1$ 
and 
$\mathcal{O}_{L_{y}}(-1):=\mathcal{O}_{\mathbb{P}^1}(-1)$. 
\end{lem}

\begin{proof}
When $\dim Y=2$, 
the same statement is stated and proven 
by Nakajima and Yoshioka 
in \cite[Proposition 1.9]{ny11b}. 
However, their proof does not work 
in the higher dimension. 
Hence we give the another proof 
which works in any dimension. 

Assume that $E \in \Per(X/Y)$. 
Then by the definition of $\Per(X/Y)$, 
we have 
$\Hom(E, \mathcal{C}^0)=0$. 
In particular, we have 
$\Hom(E, \mathcal{O}_{L_{y}}(-1))=0$. 

For the converse, we have to show the 
following two things: 
\begin{enumerate}
\item[(a)] $\Hom(E, \mathcal{C}^0)=0$, 
\item[(b)] $\dR^1f_{*}E=0$. 
\end{enumerate}

First we prove (a). 
We need to show that 
$
\Hom
\left(
E, \pi^{*}M \otimes \mathcal{O}(D)
\right)
$
vanishes for all $M \in \Coh(C)$. 
Take an element 
$\psi \in \Hom
\left(
E, \pi^{*}M \otimes \mathcal{O}(D)
\right)$. 
For each point $y \in C$, 
consider the restriction 
\[
\psi|_{L_{y}} \colon E|_{L_{y}} 
\to \pi^{*}(M|_{\{y\}}) \otimes \mathcal{O}_{L_{y}}(D) 
\cong \mathcal{O}_{L_{y}}(-1)^{\oplus k}
\]
($k \in \mathbb{Z}_{\geq0}$). 
By our assumption, 
$\psi|_{L_{y}}$ is a zero map 
for all $y \in C$. 
Hence $\psi$ itself must be zero. 
This proves (a). 

Next we prove (b). 
By the formal function theorem, 
it is enough to show that 
for every $y \in C$ and $n \in \mathbb{N}$, 
$H^1(L_{y, n}, E_{y, n})=0$. 
Here, $L_{y, n}:=X \times_{Y} \Spec \mathcal{O}_{Y, y}/m_{y}^n$ 
and $E_{y, n}:=E|_{L_{y, n}}$. 

We argue by induction on $n$. 
First let $n=1$. 
To obtain a contradiction, 
suppose that there exists $y \in C$ 
such that $H^1(L_{y}, E_{y, 1}) \neq 0$. 
Let us consider the exact sequence 
\[
0 \to (E_{y, 1})_{tor} \to E_{y, 1} \to (E_{y, 1})_{fr} \to 0, 
\]
where $(E_{y, 1})_{tor}$
(resp. $(E_{y, 1})_{fr}$) 
is the torsion part 
(resp. torsion free part)
of $E_{y, 1}$. 
Since $L_{y} \cong \mathbb{P}^1$, 
there exist integers $a_{i} \in \mathbb{Z}$ ($i=1, \cdots, l$) 
such that 
$(E_{y, 1})_{fr} \cong \bigoplus_{i=1}^{n} \mathcal{O}_{L_{y}}(a_{i})$. 
Since 
$\bigoplus_{i=1}^{n} H^1(L_{y}, \mathcal{O}_{L_{y}}(a_{i}))
\cong H^1(L_{y}, E_{y, 1}) \neq 0$, 
there exists $i_{0}$ such that $a_{i_{0}} \leq -2$. 
On the other hand, by the surjection 
$E \to E_{y, 1} \to \mathcal{O}_{L_{y}}(a_{i_{0}})$, 
we have 
\[
\Hom(\mathcal{O}_{L_{y}}(a_{i_{0}}), \mathcal{O}_{L_{y}}(-1)) 
\subset 
\Hom(E, \mathcal{O}_{L_{y}}(-1))=0, 
\]
which is a contradiction. 
Hence when $n=1$, we have 
$H^1(L_{y}, E_{y, 1})=0$
for every $y \in C$. 

Next assume that for a fixed integer $n \in \mathbb{N}$, 
$H^1(L_{y, n}, E_{y, n})=0$ holds ($y \in C$). 
Consider the exact sequence 
\[
0 \to m_{y}^n/m_{y}^{n+1} \to \mathcal{O}_{Y, y}/m_{y}^{n+1} 
\to \mathcal{O}_{Y, y}/m_{y}^{n} \to 0. 
\]
Note that 
$m_{y}^n/m_{y}^{n+1} \cong (\mathcal{O}_{Y, y}/m_{y})^{\oplus k}$ 
for some $k \in \mathbb{N}$. 
Applying the functor 
$f^{*}(-) \otimes E$ to the above exact sequence, 
we have the exact sequence 
\[
E_{y, 1}^{\oplus k} \to E_{y, n+1} \to E_{y, n} \to 0. 
\]
Split this exact sequence into two short exact sequences: 
\begin{align*}
&0 \to M \to E_{y, 1}^{\oplus k} \to K \to 0, \\
&0 \to K \to E_{n+1} \to E_{n} \to 0. 
\end{align*}
From the first exact sequence, we have the exact sequence 
\[
0=H^1(X, E_{y, 1})^{\oplus k} \to H^1(X, K) \to H^2(X, M)=0. 
\]
Note that 
$H^1(X, E_{y, 1})=0$ follows from the argument 
of $n=1$ case, while 
$H^2(X, M)=0$ holds since $\dim\Supp(M) \leq 1$. 
Hence we also have $H^1(X, K)=0$. 
Then by the second exact sequence and the induction hypothesis, 
we conclude that 
$H^1(E_{y, n+1})=0$ as required.  
\end{proof}

\begin{cor}
\label{cor of criterion}
Let $E \in \Per(X/Y) \cap \Coh(X)$ 
be a perverse coherent sheaf. 
Then $E(-D)$ is also a perverse coherent sheaf. 
\end{cor}

\begin{proof}
Take a point $y \in C$. 
Noting the isomorphism 
$\mathcal{O}_{L_{y}}(D) \cong \mathcal{O}_{L_{y}}(-1)$, 
we have 
\[
\Hom\left(E(-D), \mathcal{O}_{L_{y}}(-1)
\right)
=\Hom\left(E, \mathcal{O}_{L_{y}}(-2)
\right)
\subset \Hom\left(E, \mathcal{O}_{L_{y}}(-1)
\right)=0.
\]
Hence by Lemma \ref{criterion for per}, 
we have $E(-D) \in \Per(X/Y)$. 
\end{proof}


\subsection{Moduli space of $m$-stable sheaves}
In this subsection, we recall 
the notion of $m$-stability 
and the moduli space of $m$-stable sheaves 
introduced by Nakajima and Yoshioka 
in their paper \cite{ny11b}. 
Let $H$ be an ample divisor on $Y$, 
$v=(v_{0}, v_{1}, \cdots) \in H^{2*}(X; \mathbb{Q})$ 
such that 
$v_{0}>0$ and 
$\gcd(v_{0}, v_{1}.f^*H^{n-1})=1$. 

\begin{defin}
\label{m-stability}
Let $E \in \Coh(X)$ 
be a coherent sheaf. 
\begin{enumerate}
\item We say that 
$E$ is {\it $0$-stable} if 
$E \in \Per(X/Y)$ and 
$f_{*}E \in \Coh(Y)$ is $\mu_{H}$-stable. 

\item Let $m \in \mathbb{Z}_{>0}$ 
be a positive integer. 
We say that 
$E$ is {\it $m$-stable} if 
$E(-mD)$ is $0$-stable. 
\end{enumerate}
\end{defin}

\begin{thm}[{\cite[Theorem 2.9]{ny11b}}]
Let $m \in \mathbb{Z}_{\geq 0}$. 
Then there exists the projective coarse moduli scheme $M^m(v)$ 
of $m$-stable sheaves with Chern character $v$. 
\end{thm}

\begin{rmk}
In \cite[Section 2]{ny11b}, 
the notion of $m$-stability 
and the existence of the coarse moduli space 
are discussed without assuming 
$\gcd(v_{0}, v_{1}.f^*H^{n-1})=1$. 
However, in the following, 
we use this assumption 
almost everywhere. 
In particular, the following fact 
will be used frequently: 
for $E \in \Per(X/Y) \cap \Coh(X)$ 
with $\ch(E)=v$, 
$f_{*}E$ is $\mu_{H}$-semistable 
if and only if it is $\mu_{H}$-stable. 
\end{rmk}

\section{Wall-crossing}
\label{wall crossing}
In this section, 
we always fix an ample divisor $H$ 
on $Y$ and the Chern character 
$v=(v_{0}, v_{1}, \cdots) 
\in H^{2*}(X; \mathbb{Q})$ 
of a sheaf 
such that $v_{0} > 0$ and 
$\gcd(v_{0}, v_{1},f^*H^{n-1})=1$. 
We will describe 
the difference between $m$-stability 
and $(m+1)$-stability. 
To do that, we may assume $m=0$ 
since we have an isomorphism 
\[
(-) \otimes \mathcal{O}(-mD) \colon 
M^m(v) \to M^0(v.e^{-mD}). 
\]


\subsection{Torsion pairs}

To construct the diagram (\ref{zig-zag}), 
we use {\it torsion pairs} (cf. \cite{hrs96}). 

\begin{defin}
\label{torsion pair}
Let $\mathcal{A}$ be an Abelian category, 
$\mathcal{T}, \mathcal{F} \subset \mathcal{A}$ 
full subcategories of $\mathcal{A}$. 
Then the pair $(\mathcal{T}, \mathcal{F})$ 
is a {\it torsion pair} 
on $\mathcal{A}$ 
if the following two conditions hold: 

\begin{enumerate}
\item $\Hom\left(
\mathcal{T}, \mathcal{F}
\right)=0$. 

\item For every object $E \in \mathcal{A}$, 
there exists objects 
$T \in \mathcal{T}$, 
$F \in \mathcal{F}$, 
and an short exact sequence 
\[
0 \to T \to E \to F \to 0. 
\]
\end{enumerate}
\end{defin}

Note that by the property (1), 
the exact sequence (2) is 
unique up to isomorphism. 
First we recall the following easy property of torsion pairs. 

\begin{lem}
\label{property tor pair}
Let $(\mathcal{T}, \mathcal{F})$ be a torsion pair 
on an Abelian category $\mathcal{A}$. 
Then $\mathcal{T}$ is closed under taking quotients. 
\end{lem}
\begin{proof}
Take an element $E \in \mathcal{T}$ 
and a surjective map 
$E \twoheadrightarrow Q$ 
in $\mathcal{A}$. 
By the definition of a torsion pair, 
there exists an exact sequence 
\[
0 \to T \to Q \to F \to 0
\]
with $T \in \mathcal{T}$, 
$F \in \mathcal{F}$. 
Then we get a surjective map 
$E \twoheadrightarrow Q \twoheadrightarrow F$. 
On the other hand, we have 
$\Hom\left(
\mathcal{T}, \mathcal{F}
\right)=0$ 
by definition. 
Hence we must have $F=0$, i.e. 
$Q=T \in \mathcal{T}$. 
\end{proof}

In this paper, 
we use the following two torsion pairs 
on $\Coh(X)$. 

\begin{defin}
\label{torsion pair for per}
Define the full subcategories of $\Coh(X)$ as 
\begin{align*}
&\mathcal{T}:=\left\{
 T \in \Coh(X) : 
 \dR^1f_{*}T=0, 
 \Hom(T, \mathcal{C}^0)=0
  \right\}, \\
&\mathcal{F}:=\left\{
 F \in \Coh(X) : 
 f_{*}F=0 
 \right\}. 
\end{align*}
\end{defin}

\begin{defin}
\label{torsion pair for pertilde}
We define the full subcategories 
$\mathcal{T}_{D}, \mathcal{F}_{D} 
\subset \Coh(X)$ as follows: 
\begin{align*}
&\mathcal{T}_{D}:=
\left\{
T \in \Coh(X) : 
T(-D) \in \Per(X/Y), 
\Supp(T) \subset D
\right\}, \\ 
&\mathcal{F}_{D}:=
\left\{
F \in \Coh(X) : 
\Hom(\mathcal{T}_{D}, F)=0
\right\}. 
\end{align*} 
\end{defin}

\begin{lem}
The pairs 
$(\mathcal{T}, \mathcal{F})$ and 
$(\mathcal{T}_{D}, \mathcal{F}_{D})$ 
are torsion pairs on $\Coh(X)$. 
\end{lem}

\begin{proof}
The assertion for the pair 
$(\mathcal{T}, \mathcal{F})$ 
has been proved 
for example in \cite[Lemma 1.3]{ny11b}. 
For the pair $(\mathcal{T}_{D}, \mathcal{F}_{D})$, 
it is enough to show that 
the subcategory 
$\mathcal{T}_{D} \subset \Coh(X)$ 
is closed under taking quotients and extensions 
by \cite[Lemma 2.15]{tod13a}, 
since $\Coh(X)$ is noetherian. 
The conditions $\Supp(T) \subset D$ 
and $\Hom(T(-D), \mathcal{C}^0)=0$ 
are clearly closed under taking quotients and extensions. 
Furthermore, since the morphism $f \colon X \to Y$ 
has relative dimension one, 
we can also see that the condition 
$\dR^1f_{*}T(-D)=0$ 
is closed under taking quotients and extensions. 
\end{proof}

\begin{rmk}
\label{rmk tilper}
Recall that we have 
$\Per(X/Y)=\left\langle
 \mathcal{F}[1], \mathcal{T}
 \right\rangle$. 
Similarly, we define 
\[
\widetilde{\Per}(X/Y):=
\left\langle
 \mathcal{F}_{D}[1], \mathcal{T}_{D}
 \right\rangle.
 \]
By the general theory of torsion pairs and tilting, 
the category $\widetilde{\Per}(X/Y)$ 
is the heart of a bounded t-structure 
on $D^b(X)$ (cf. \cite{hrs96}). 
\end{rmk}

\begin{lem}
\label{criterion for tor free}
We have $\mathcal{C}^0 \subset \mathcal{T}_{D}$. 
Furthermore, for a coherent sheaf $E \in \Coh(X)$ 
which is torsion free outside $D$, 
the following statements hold. 
\begin{enumerate}
\item The sheaf $f_{*}(E(-D))$ 
is torsion free if and only if 
$\Hom(\mathcal{T}_{D}, E)=0$, i.e. 
$E \in \mathcal{F}_{D}$. 

\item If the sheaf $f_{*}(E(-D))$ is torsion free, 
then $f_{*}(E)$ is also torsion free. 
\end{enumerate}
\end{lem}

\begin{proof}
For the first assertion, 
recall from Lemma \ref{calc} that 
we have 
$\mathcal{C}^0=f^*\Coh(C) \otimes \mathcal{O}(D)$. 
From this, we can see that 
$\Supp(T) \subset D$ 
for any element $T \in \mathcal{C}^0$. 
Moreover, we have 
$\mathcal{C}^0 \otimes \mathcal{O}(-D) 
\cong f^*\Coh(C) \subset \mathcal{T}$. 
This proves the inclusion 
$\mathcal{C}^0 \subset \mathcal{T}_{D}$. 

(1) First assume that 
$\Hom(\mathcal{T}_{D}, E)=0$. 
Since we assume $E$ is torsion free outside $D$, 
it is enough to show that 
$\Hom(\Coh(C), f_{*}(E(-D)))=0$. 
We can compute as  
\[
\begin{split}
\Hom(\Coh(C), f_{*}(E(-D))) 
&\cong \Hom(f^{*}\Coh(C), E(-D)) \\
&=\Hom(f^{*}\Coh(C) \otimes \mathcal{O}(D), E) \\
&=0. 
\end{split}
\]
Note that the last equality holds by 
the inclusion 
$\mathcal{C}^0 
\subset \mathcal{T}_{D}$. 

For the converse, assume that $f_{*}(E(-D))$ is torsion free. 
Let $T \in \mathcal{T}_{D}$. Then 
\[
\begin{split}
\Hom(T, E)
&=\Hom(T(-D), E(-D)) \\
&\subset \Hom(f^{*}f_{*}(T(-D)), E(-D)) \\
&\cong \Hom(f_{*}(T(-D)), f_{*}(E(-D))) \\ 
&=0. 
\end{split}
\]
Note that the adjoint map 
$f^{*}f_{*}(T(-D)) \to T(-D)$ 
is surjective since $T(-D) \in \Per(X/Y)$ 
by definition. 
Hence we have the inclusion 
$\Hom(T(-D), E(-D)) 
\subset \Hom(f^{*}f_{*}(T(-D)), E(-D))$. 
Note also that the last equality holds 
since $f_{*}(T(-D))$ is torsion. 

(2) By (1), it is enough to show that 
$\Hom(\mathcal{T}_{D}, E(D))=0$. 
Let $T \in \mathcal{T}_{D}$. Then 
we have $T(-D) \in \mathcal{T}_{D}$ 
by Corollary \ref{cor of criterion}. Hence 
\[
\Hom(T, E(D)) 
=\Hom(T(-D), E)=0.  
\]
Here, the last equality holds again by (1). 
\end{proof}


\subsection{From $0$-stability to $1$-stability}

\begin{prop}
\label{0-st to 1-st}
Let $E^{-} \in \Coh(X) \cap \Per(X/Y)$ 
be a perverse coherent sheaf 
with Chern character $\ch(E^-)=v$. 
Let 
\[
0 \to T \to E^{-} \to F \to 0
\]
be the unique exact sequence in $\Coh(X)$ with 
$T \in \mathcal{T}_{D}, F \in \mathcal{F}_{D}$. 
Then the following hold:  
\begin{enumerate}
\item If $f_{*}E^{-}$ is torsion free, then $T \in \mathcal{C}^0$. 
\item If $E^{-}$ is $0$-stable, 
then $F$ is $0$-stable and 
$1$-stable. 
\end{enumerate}
\end{prop}

\begin{proof}
First we prove (1). 
Since $T(-D) \in \Per(X/Y)$, 
we have 
$\dR^1f_{*}T=0$ 
by Lemma \ref{r1 per}. 
On the other hand, 
we have the injection 
$0 \to f_{*}T \to f_{*}E^{-}$
in $\Coh(Y)$. 
Since $T$ is supported on $D$ and $f_{*}E^{-}$ is 
torsion free, we have $f_{*}T=0$. 
This proves (1). 

Next assume that $E^{-}$ is $0$-stable. 
Since we have a surjection 
$E^{-} \to F \to 0$ 
in $\Coh(X)$, 
we have 
$F \in \Per(X/Y) \cap \Coh(X)=\mathcal{T}$ 
by Lemma \ref{property tor pair}. 
Furthermore,  the isomorphism 
$f_{*}E^{-} \cong f_{*}F$ 
implies 
$F$ is $0$-stable. 
It remains to show that 
$F$ is $1$-stable. 
By Corollary \ref{cor of criterion}, we have 
$F(-D) \in \Per(X/Y)$. 

\begin{claim}
$f_{*}(F(-D))$ is torsion free. 
\end{claim}

\begin{proof}
Let $M \in \Coh(C)$. 
We must show that 
$\Hom(M, f_{*}(F(-D)))=0$. 
By Lemma \ref{calc} and Lemma \ref{criterion for tor free}, 
we have 
$f^*M \otimes \mathcal{O}(D) 
\in \mathcal{C}^0 
\subset \mathcal{T}_{D}$. 
Hence by adjunction, we get 
\[
\Hom(M, f_{*}(F(-D))) 
\cong 
\Hom(f^{*}M \otimes \mathcal{O}(D), F)
=0. 
\]
\end{proof}

Now since $f_{*}(F(-D)) \to f_{*}F$ is an isomorpism 
outside $C$ and they are torsion free, 
it follows that 
$0 \to f_{*}(F(-D)) \to f_{*}F$ 
is injective in $\Coh(Y)$. 
Furthermore, by the facts that 
$f_{*}F$ is $\mu$-stable and 
$\mu(f_{*}(F(-D)))=\mu(f_{*}F)$, 
we conclude that $f_{*}(F(-D))$ is 
$\mu$-stable. 
This means that $F$ is 1-stable. 
\end{proof}

\begin{prop}
\label{0-st to 0-st}
Let $F \in M^0(v) \cap M^1(v)$ 
be a $0$-stable and $1$-stable sheaf 
with Chern character $\ch(F)=v$. 
Let $T \in \mathcal{C}^0$ with a surjective map 
$\psi \colon F \to T[1] \to 0$ 
in $\Per(X/Y)$. 
Let $E^{-}:=\Ker\psi \in \Per(X/Y)$. 
Then $E^{-}$ is $0$-stable but not $1$-stable. 
\end{prop}

\begin{proof}
By shifting the exact triangle 
$E^- \to F \to T[1]$, 
we get an exact triangle 
$T \to E^- \to F$. 
Since $T, F \in \Coh(X)$, 
we also have 
$E^- \in \Coh(X)$. 
Furthermore, since 
$E^{-} \in \Per(X/Y)$ and 
$f_{*}E^{-} \cong f_{*}F$, 
$E^{-}$ is $0$-stable. 
On the other hand, 
since $T \in \mathcal{C}^0$ and 
$\Hom(T, E^{-}) \neq 0$, 
$E^{-}$ is not $1$-stable by 
Lemma \ref{criterion for tor free}. 
\end{proof}


\subsection{From $1$-stability to $0$-stability}

\begin{prop}
\label{1-st to 0-st}
Let $E^+ \in M^1(v)$ 
be a $1$-stable sheaf with 
Chern character $\ch(E^+)=v$. 
Let 
\[
0 \to F \to E^{+} \to T \to 0
\]
be the unique exact sequence in $\Coh(X)$ with 
$F \in \mathcal{T}$, $T \in \mathcal{F}$. 
Then we have 
\begin{enumerate}
\item $T \in \mathcal{C}^0$. 
\item $F$ is $0$-stable and $1$-stable. 
\end{enumerate}
\end{prop}

\begin{proof}
(1) By definition, we have $f_{*}T=0$. 
On the other hand, since 
$E^{+}(-D) \in \Per(X/Y) \cap \Coh(X)=\mathcal{T}$, 
it follows that 
$T(-D) \in \Per(X/Y)$ 
by Lemma \ref{property tor pair}. 
Hence by Lemma \ref{r1 per}, we get $\dR^1f_{*}T=0$. 

(2) By definition, $F \in \Per(X/Y)$. 
By Corollary \ref{cor of criterion}, 
we also have $F(-D) \in \Per(X/Y)$. 
First we will show that $F$ is $1$-stable. 
We have an exact sequence 
\[
0 \to f_{*}(F(-D)) \to f_{*}(E^{+}(-D)) \to f_{*}(T(-D))
\]
in $\Coh(Y)$. 
Since $\Supp(f_{*}(T(-D))) \subset C$ 
and $f_{*}(E^{+}(-D))$ is $\mu$-stable, 
it follows that $f_{*}(F(-D))$ is 
also $\mu$-stable. 
We conclude that $F$ is $1$-stable. 

It remains to show that $F$ is $0$-stable. 
As in the argument of the previous proposition, 
it is enough to show that 
$f_{*}F$ is torsion free. 
But that follows from Lemma \ref{criterion for tor free} (3).  
\end{proof}

\begin{prop}
\label{1-st to 1-st}
Let $F \in M^0(v) \cap M^1(v)$ 
be a $0$-stable and $1$-stable sheaf 
with Chern character $\ch(F)=v$. 
Let $T \in \mathcal{C}^0$ with an injective map 
$0 \to T \to F[1]$ in 
$\widetilde{\Per}(X/Y)$. 
Let $E^{+}[1]$ be its cokernel. 
Then $E^{+}$ is $1$-stable but not $0$-stable. 
\end{prop}

\begin{proof}
First note that $F \in \mathcal{F}_{D}$ 
and hence $F[1] \in \widetilde{\Per}(X/Y)$. 
Furthermore, $\mathcal{F}_{D}[1]$ is a torsion part of 
the torsion pair $(\mathcal{F}_{D}[1], \mathcal{T}_{D})$
on $\widetilde{\Per}(X/Y)$. 
Hence $\mathcal{F}_{D}[1]$ is closed under taking quotients 
in the abelian category $\widetilde{\Per}(X/Y)$. 
In particular, $E^{+} \in \mathcal{F}_{D}$ and hence 
$f_{*}(E^{+}(-D))$ is torsion free. 
Applying $f_{*}$ to the exact sequence 
\[
0 \to F \to E^{+} \to T \to 0
\]
in $\Coh(X)$, we have an injection 
$0 \to f_{*}(F(-D)) \to f_{*}(E^{+}(-D))$ 
which is isomorphism outside $C$. 
Hence the $\mu$-stability of $f_{*}(F(-D))$ implies that 
$f_{*}(E^{+}(-D))$ is $\mu$-stable. 
On the other hand, since $T \in \mathcal{C}^0$ and 
$\Hom(E^{+}, T) \neq 0$, 
we have $E^{+} \not \in \Per(X/Y)$. 
In particular, $E^{+}$ is not $0$-stable.

\end{proof}


\subsection{Set-theoretical wall-crossing}

Define 
$\mathcal{S}:=
\{
\ch(T) : T \in \mathcal{C}^0
\}. 
$
First we define the following two notions. 
Note that for a $0$-stable and $1$-stable sheaf $F$, 
we have $F \in \Per(X/Y)$ by definition. 
Furthermore, by Lemma \ref{criterion for tor free}, 
we have $F \in \mathcal{F}_{D}$ and thus 
$F[1] \in \widetilde{\Per}(X/Y)$ 
(see Remark \ref{rmk tilper} 
for the definition of 
$\widetilde{\Per}(X/Y)$). 
\begin{defin}
\label{pquot}
Let $F$ be a $0$-stable and $1$-stable sheaf 
and $\beta \in \mathcal{S}$. 
We define 
\begin{align*}
&\PSub(F, \beta):=
\{
0 \to E \to F \to T[1]  \to 0 \mbox{ exact in } \Per(X/Y) : 
T \in \mathcal{C}^0, \ch(T)=\beta
\}, \\
&\widetilde{\PQuot}(F, \beta):=
\{
0 \to T \to F[1] \to E[1] \to 0 \mbox{ exact in } \widetilde{\Per}(X/Y) : 
T \in \mathcal{C}^0, \ch(T)=\beta
\}. 
\end{align*}
\end{defin}

Summarizing the results in the previous subsections, 
we get: 

\begin{prop}
\label{set}
We have a diagram of sets 
\[
\xymatrix{
&M^0(v) \ar[rd]_{\xi_{0}^{-}} & &M^1(v) \ar[ld]^{\xi_{0}^{+}} \\
& &M^{0, 1}(v) & 
}
\]
such that 
\begin{itemize}
\item 
$M^{0, 1}(v):=\coprod_{\beta \in \mathcal{S}} 
(M^0(v-\beta) \cap M^1(v-\beta))$, 
\item 
the fibre of $F \in M^0(v-\beta) \cap M^1(v-\beta)$ 
over $\xi_{0}^{-}$ is $\PSub(F, \beta)$, 
\item the fibre of $F \in M^0(v-\beta) \cap M^1(v-\beta)$ 
over $\xi_{0}^{+}$ is $\widetilde{\PQuot}(F, \beta)$. 
\end{itemize}
\end{prop}
\begin{proof}
We define the map $\xi^-_{0}$ as follows: 
Take an element $E^- \in M^0(v)$, 
let 
$0 \to T \to E^- \to F \to 0$ 
be the canonical exact sequence 
with $T \in \mathcal{T}_{D}$, 
$F \in \mathcal{F}_{D}$. 
Then $F$ is an element of $M^{0, 1}(v)$ 
by Proposition \ref{0-st to 1-st}. 
Hence we define 
$\xi^{-}_{0}(E^-):=F$. 
By the converse construction given in 
Proposition \ref{0-st to 0-st}, 
the fiber of $\xi^{-}_{0}$ is given by 
$\PSub(F, \beta)$. 

The map $\xi_{0}^+$ is defined as follows: 
for an element $E^+ \in M^1(v)$, 
we have an exact sequence 
$0 \to F \to E^+ \to T \to 0$ 
with $F \in \mathcal{T}$ and 
$T \in \mathcal{F}$. 
Now we define as $\xi^+_{0}(E^+):=F$. 
Then the assersion follows from 
Proposition \ref{1-st to 0-st} and 
Proposition \ref{1-st to 1-st}. 
\end{proof}

In Section \ref{scheme str}, 
we will construct the scheme-theoretic wall-crossing diagram. 


\section{Moduli space of Gieseker stable sheaves}
\label{Gieseker}
In this section, we will see the 
relationship between 
$m$-stability and 
the Gieseler stability on both $X$ and $Y$. 

\subsection{Moduli space of Gieseker stable sheaves on $Y$}
Take 
$v=(v_{0}, v_{1}, \cdots) 
\in H^{2*}(X, \mathbb{Q})$ 
with $\gcd(f^{*}H^{n-1}.v_{1}, v_{0})=1$, 
and let 
$w:=f_{*}(v.\td_{X}).\td_{Y}^{-1}$. 
Denote  by $M^{H}(w)$ the moduli space of 
Gieseker $H$-stable sheaves on $Y$ 
with Chern character $v'$. 
Then we have a morphism of schemes: 
\[
\xi \colon M^0(v) \to M^{H}(w), \ 
E \mapsto \dR f_{*}E=f_{*}E. 
\]

Similarly, we have 
\[
\xi^{+} \colon M^1(v) \to M^{H}(w), \ 
E \mapsto \dR f_{*}E=f_{*}E. 
\]

Note that $\xi^{+}$ is well-defined. 
Indeed, take an element $E \in M^1(v)$. 
Then by Lemma \ref{criterion for tor free}, 
$f_{*}E$ is torsion free. 
Furthermore, since $f_{*}E(-D)$ is 
$\mu$-stable and is isomorphic to 
$f_{*}E$ in codimension $1$, 
$f_{*}E$ is also $\mu$-stable. 
Hence $\xi^{+}$ is actually a morphism 
from $M^1(v)$ to $M^{H}(w)$. 

\begin{lem}
\label{pull back in per}
Let $F \in \Coh(Y)$ be a torsion free sheaf on $Y$. 
Then we have 
$\dL f^{*}F=f^{*}F$. 
\end{lem}

\begin{proof}
First we claim that $\dL f^{*}F \in \Per(X/Y)$. 
We need to check the following three conditions
(cf. \cite[Lemma 3.2]{bri02}): 
\begin{enumerate}
\item[(a)] $\dR f_{*}(\dL f^{*}F) \in \Coh(Y)$, 
\item[(b)] $\Hom(\dL f^{*}F, \mathcal{C}^{>-1})=0$, 
\item[(c)] $\Hom(\mathcal{C}^{<-1}, \dL f^{*}F)=0$. 
\end{enumerate}

(a), (b) are clear. We will show (c). 
Let $T \in \mathcal{C}^0$, $i \geq 1$, 
$f_{!}:=\dR f_{*}(-\otimes \omega_{X})\otimes \omega_{Y}^{-1}$. 
Then we have 
\[
\Hom(T[i], \dL f^{*}F) \cong \Hom(f_{!}T[i], F). 
\]

By the description of $f_{!}$, we know that 
$f_{!}T$ is a two term complex concentrated in 
degree $0$ and $1$. 
Hence there exist $T', T'' \in \Coh(Y)$ and 
an exact triangle 
\[
T'[i] \to f_{!}T[i] \to T''[i-1]. 
\]

Applying $\Hom(-, F)$, we get an exact sequence 
\[
\Hom(T''[i-1], F) \to \Hom(f_{!}T[i], F) \to \Hom(T'[i], F). 
\]

Since $T', T''$ are torsion sheaves 
and $F$ is torsion free sheaf, 
we conclude that 
\begin{equation}
\label{vanish}
\Hom(T[i], \dL f^{*}F) \cong \Hom(f_{!}T[i], F)=0. 
\end{equation}

Hence $\dL f^{*}F \in \Per(X/Y)$. 

Next we claim that $\dL^{-1}f^{*}F \in \mathcal{C}^0$. 
Note that, if so, together with the equation (\ref{vanish}), 
we must have $\dL^{-1} f^{*}F=0$, i,e, 
$\dL f^{*}F=f^{*}F$. 
By the spectral sequence 
\[
\dR^{p}\dL^{q}f^{*}F \Rightarrow \mathcal{H}^{p+q}(\dR f_{*}\dL f^{*}F)=F, 
\]
we have 
\begin{itemize}
\item $f_{*}(\dL^{-1}f_{*}F)=0$, 
\item $0 \to \dR^1 f_{*}(\dL^{-1} f^{*}F) \to F \to f_{*}f^{*}F \to 0$ 
is exact in $\Coh(Y)$. 
\end{itemize}
Since $\dR^1 f_{*}(\dL^{-1} f^{*}F)$ is a torsion sheaf and 
we assume $F$ is torsion free, we have 
$\dR^1 f_{*}(\dL^{-1} f^{*}F)=0$. 
This proves that $\dL f^{*}F=f^{*}F$. 
\end{proof}

\begin{cor}
\label{xi is isom}
Assume that 
$v \in f^{*}H^{*}(Y ; \mathbb{Q}) \subset H^{*}(X ; \mathbb{Q})$. 
Then the morphism $\xi$ is an isomorphism. 
\end{cor}

\begin{proof}
By Lemma \ref{pull back in per}, we can define a morphism $\eta$ as 
\[
\eta \colon M^{H}(w) \to M^0(v), \ 
F \mapsto \dL f^{*}F=f^{*}F. 
\]
Then by the projection formula, we have 
$\xi \circ \eta=\id$. 
On the other hand, let $E \in M^0(v)$. 
Then we have an exact sequence 
\[
0 \to K \to f^{*}f_{*}E \to E \to 0
\]
in $\Coh(X)$, since $E \in \Per(X/Y)$. 
By Lemma \ref{property of adjoint} (4), 
we have $K \in \mathcal{C}^0$. 
In particular, 
$\ch(K) \not\in f^{*}H^{*}(Y ; \mathbb{Q})$. 
Since we assume $v \in f^{*}H^{*}(Y ; \mathbb{Q})$, 
we must have $K=0$, i.e, 
$\eta \circ \xi=\id$. 
\end{proof}


\subsection{Moduli space of Gieseker stable sheaves on $X$}

First we recall the finiteness result of walls for 
$\mu$-stability: 

\begin{prop}[cf. {\cite[Lemma 1.1.7]{sch00}}]
\label{finite wall for mu}
Let $H$ be an ample divisor on $Y$, 
take $\epsilon_{0} \in \mathbb{Q}_{>0}$ 
so that $f^{*}H-\epsilon_{0} D$ is ample on $X$. 
Let 
\[
\Delta:=
\left\{
H_{\epsilon}:=f^{*}H-\epsilon D : 
0 \leq \epsilon \leq \epsilon_{0}
\right\}. 
\]
Then there exist only finitely many walls  on $\Delta$ 
for $\mu$-stability with respect to $v$. 
\end{prop}

\begin{proof}
The argument is essentially same as 
\cite[Lemma 1.1.7]{sch00}. 
However, since $f^{*}H$ is not ample, 
we need to be a little bit careful. 
The proof needs the following two facts: 
\begin{enumerate}
\item The Bogomolov-Gieseker (BG) inequality 
for $\mu_{H_{\epsilon}}$-stable sheaves. 

\item Let $x \in \NS(X)_{\mathbb{R}}$. 
If $x.H_{\epsilon}^{n-1}=0$, 
then $-x^2.H_{\epsilon}^{n-2} \geq 0$. 
\end{enumerate}

Firstly, (1) holds since the BG inequality holds 
for all nef divisor (cf. \cite{lan04}). 
Next consider the statement (2). 
If $\epsilon \neq 0$, 
then $H_{\epsilon}$ is ample and 
the statement holds by 
the Hodge index theorem. 
Assume $\epsilon=0$, 
let 
$x_{\epsilon}:=
x-\frac{x.H_{\epsilon}^{n-1}}{D.H_{\epsilon}^{n-1}}D$. 
Then 
\[
x_{\epsilon}.H_{\epsilon}^{n-1}
=x.H_{\epsilon}^{n-1}
-\frac{x.H_{\epsilon}^{n-1}}{D.H_{\epsilon}^{n-1}}D.H_{\epsilon}^{n-1}
=0
\]
and 
$\lim_{\epsilon \to +0} x_{\epsilon}=x$. 
Since $x_{\epsilon}$ changes continuously 
with respect to $\epsilon$, 
we get the required result 
from the result of the ample case. 
\end{proof}

By Proposition \ref{finite wall for mu}, 
there exist $0 < \epsilon(v)$ such that 
for every $0 < \epsilon \leq \epsilon(v)$, 
$M^{f^{*}H-\epsilon D}(v)$ is constant. 

\begin{prop}
\label{moduli on X}
There exists an integer $m(v)>0$ such that 
for every $m \geq m(v)$, we have 
an open and closed embedding 
$
M^{f^*H-\epsilon(v)D}(v) \subset M^m(v). 
$
\end{prop}

\begin{proof}
It is enough to show the following: 
there exists $m(v)>0$ such that 
for every $m \geq m(v)$ and $E \in M^{f^{*}H-\epsilon(v)D}(v)$, 
we have $E \in M^m(v)$. 
Indeed, if so, we have an immersion 
$\Phi \colon M^{f^{*}H-\epsilon(v)D}(v) \hookrightarrow M^m(v)$. 
By the openness of $\mu$-stability, 
$\Phi$ is an open immersion. 
On the other hand, since both 
$M^{f^{*}H-\epsilon(v)D}(v)$ and 
$M^m(v)$ are projective, 
$\Phi$ is also a closed immersion. 

First of all, 
since $M^{f^*H-\epsilon(v)D}(v)$ is bounded, 
there exists a fixed sheaf $U \in \Coh(X)$ 
such that 
for every $E \in M^{f^{*}H-\epsilon(v)D}(v)$, 
we have a surjective map
$U \twoheadrightarrow E$ 
(cf. \cite[Lemma 1.7.6]{hl97}). 
On the other hand, since $-D$ is $f$-ample, 
there exists an integer $m(v) > 0$ 
such that for every 
$m \geq m(v)$, 
the adjoint map 
$f^*f_{*}(U(-mD)) \to U(-mD)$ 
is surjective. 
Then we have a commutative diagram 
\[
\xymatrix{
&f^*f_{*}(U(-mD)) \ar@{>>}[d] \ar@{>>}[r] &U(-mD) \ar@{>>}[d] \\
&f^*f_{*}(E(-mD)) \ar^{\alpha}[r] &E(-mD). 
}
\]
Hence the adjoint map $\alpha$ 
is also surjective, i.e. $E(-mD) \in \Per(X/Y)$. 

It remains to show that for $m \geq m(v)$, 
$f_{*}(E(-mD))$ is $\mu_{H}$-stable. 
Note that $f_{*}(E(-mD))$ is torsion free 
since so is $E$. 
Take a non-zero proper subsheaf 
$F \subset f_{*}(E(-mD))$, 
and let 
$\phi \colon f^*F \to E(-mD)$ 
be the corresponding map. 
Taking its cone 
$T:=\Cone(\phi)$, 
we have 
\begin{align*}
&0 \to \mathcal{H}^{-1}(T) \to f^*F \to Q \to 0, \\
&0 \to Q \to E(-mD) \to \mathcal{H}^0(T) \to 0. 
\end{align*}

Here, $\mathcal{H}^i(T)$ denotes the 
cohomology with respect to 
the heart $\Coh(X)$ for each integer 
$i \in \mathbb{Z}$. 
Note that  since $\phi$ is injective 
outside $D$, 
$\Supp(\mathcal{H}^{-1}(T)) \subset D$. 
Hence we can write 
$\ch(\mathcal{H}^{-1}(T))=(0, lD, \cdots)$ 
with $l \geq 0$. 
Now by the $\mu_{f^*H-\epsilon(v)D}$-stability 
of $E(-mD)$, we have 
\[
\mu_{f^*H-\epsilon D}(Q) < \mu_{f^*H-\epsilon D}(E(-mD))
\]
for all $0 <\epsilon \leq \epsilon(v)$. 
Taking the limit $\epsilon \to +0$, 
we have 
\[
\mu_{H}(F)=\mu_{f^*H}(f^*F)
=\mu_{f^*H}(Q) 
\leq \mu_{f^*H}(E(-mD))
=\mu_{H}(f_{*}(E(-mD))). 
\]

This shows that $E \in M^m(v)$. 
\end{proof}

\begin{rmk}
As mentioned in the introduction, when $n=\dim Y=2$, 
the embedding given in Proposition \ref{moduli on X} 
is actually an isomorphism (see Theorem \ref{ny}). 
When $n \geq 3$, we do not know whether 
the inclusion is isomorphism or not in general. 
However, see the following example. 
\end{rmk}

\begin{ex}
\label{ex in hilb}
Let $n=3$, $v=(1, 0, 0, -k) \in H^{2*}(X; \mathbb{Q})$ 
with $k \in \mathbb{Z}_{>0}$. 
We claim that for a sufficiently large integer 
$m \in \mathbb{Z}$, we have an isomorphism between 
$M^m(v)$ and $M^{f^*H-\epsilon D}(v)$. 
By Proposition \ref{moduli on X}, it is enough to show that 
every $m$-stable sheaf $E \in M^m(v)$ is torsion free. 
Let us consider the exact sequence 
\[
0 \to E_{tor} \to E \to E_{fr} \to 0, 
\]
where $E_{tor}$ (resp. $E_{fr}$) is 
the torsion (resp. torsion free) part of E. 

First assume that $\dim \Supp (E_{tor})=2$. 
Take a general member $A$ 
in the linear system $| H |$ and 
let $f_{A} \colon A':=f^{-1}(A) \to A$. 
We may assume that 
\begin{enumerate}
\item $(f_{*}E(-mD))|_{A} \cong f_{A*}(E(-mD)|_{A'})$ 
is torsion free. 

\item $E_{tor}|_{A'} \neq 0$. 
\end{enumerate}
This means $E|_{A'}$ is an $m$-stable sheaf on 
the smooth projective surface $A'$ 
which has a non-trivial torsion part. 
Hence by Theorem \ref{ny} (3), 
$m$ is bounded from above. 

We may now assume that $\dim \Supp(E_{tor})=1$. 
Since $f_{*}(E(-mD))$ is torsion free, 
we can see that 
\[
E_{tor} \in 
\left\langle 
\mathcal{O}_{L_{y}}(a_{y}) : 
y \in C, \quad 
a_{y}+m <0
\right\rangle. 
\]
In particular, we have 
$\ch(E_{fr})
=\left(1, 0, -\sum_{i=1}^{l}n_{i}L_{y_{i}}, 
-k-\sum_{i=1}^{l}n_{i}(a_{i}+\frac{1}{2})
\right)$, 
where $y_{i} \in C$, $a_{i}+m <0$. 
Now consider the exact sequence 
\[
0 \to E_{fr} \to (E_{fr})^{DD}\cong \mathcal{O}_{X} \to \mathcal{O}_{Z} \to 0. 
\]
Here we use the speciality of our choice of $v$, 
which implies that 
$(E_{fr})^{DD} \cong \mathcal{O}_{X}$ 
for any $E \in M^m(v)$. 
Since $\dR f_{*}\mathcal{O}_{Z}=f_{*}\mathcal{O}_{Z}$ 
is zero dimensional, 
the Riemann-Roch theorem yeilds 
\[
0 \leq \chi(\dR f_{*}\mathcal{O}_{Z}) 
=k+\sum_{i=1}^{l}n_{i}(a_{i}+\frac{1}{2}). 
\]
Since $a_{i}+m <0$, we get the inequality 
$k \geq m(\sum_{i=1}^{l}n_{i})$ 
which bounds $m$ from above. 
\end{ex}


\section{Scheme structure on $M^{0, 1}(v)$}
\label{scheme str}

In this section, we define the scheme structure on $M^{0, 1}(v)$ 
connecting $M^0(v)$ and $M^1(v)$. 
Recall that we have constructed a set-theoretic diagram 
in Proposition \ref{set}: 
\begin{equation}
\label{set theoretic}
\xymatrix{
&M^0(v) \ar[rd]_{\xi_{0}^{-}} & &M^1(v) \ar[ld]^{\xi_{0}^{+}} \\
& &M^{0, 1}(v). & 
}
\end{equation}

On the other hand, we have a scheme-theoretic diagram 
\begin{equation}
\label{scheme theoretic}
\xymatrix{
&M^0(v) \ar[rd]_{\xi=f_{*}} & &M^1(v) \ar[ld]^{\xi^{+}=f_{*}} \\
& &M^{H}(w). & 
}
\end{equation}

In the following, we will show that these two diagrams are 
essentially same.

\begin{prop}
\label{xi vs xi0}
\begin{enumerate}
\item The morphism 
\begin{equation}
\label{xi on intersection}
\xi|_{M^0(v) \cap M^1(v)} \colon 
M^0(v) \cap M^1(v) \to M^{H}(w) 
\end{equation} 
is an immersion. 

\item We can identify 
$\xi_{0}^{-}(M^0(v))$ with 
$\xi(M^0(v))$. 

\item Under the identification (2), 
we have $\xi=\xi^{-}_{0}$. 

\end{enumerate}
\end{prop}

\begin{proof}
(1) Since $M^{H}(w)$ is a projective scheme, 
it is enough to show that 
the morphism is injective and it induces 
injection between tangent spaces. 

First we will show that 
the morphism (\ref{xi on intersection}) is injective. 
Let $E, E' \in M^0(v) \cap M^1(v)$ 
with $\phi \colon f_{*}E' \cong f_{*}E$. 
Then we have the commutative diagram
\begin{equation}
\label{ex seq}
\xymatrix{
&0 \ar[r] & K' \ar[r] \ar@{.>}[d]& f^{*}f_{*}E' \ar[r] \ar[d]^{f^{*}(\phi)}& E' \ar[r] \ar@{.>}[d]^{\psi} &0 \\
&0 \ar[r] & K \ar[r] & f^{*}f_{*}E \ar[r] & E \ar[r] &0. 
}
\end{equation}
Note that since $K' \in \mathcal{C}^0$ and 
$E$ is $1$-stable, we have 
$\Hom(K', E)=0$. 
Hence the composition 
$K' \to f^{*}f_{*}E' \to f^{*}f_{*}E$ 
factors through 
$K' \to K \to f^{*}f_{*}E$. 
As a result, we get an exact sequence 
\[
\xymatrix{
&0 \ar[r] &T \ar[r] &E' \ar[r]^{\psi} &E \ar[r] &0. 
}
\]

Applying $\dR f_{*}$ to the diagram (\ref{ex seq}), 
we have $\dR f_{*}(\psi)=\phi \colon f_{*}E' \cong f_{*}E$ 
and hence $T \in \mathcal{C}^0$. 
Since $E'$ is $1$-stable, $T$ must be $0$, i.e. 
$\psi \colon E' \cong E$. 

It remains to show that for every 
$E \in M^0(v) \cap M^1(v)$, 
there exists an inclusion between tangent spaces 
\[
\xymatrix{
&T_{[E]}M^0(v)=\Ext^1(E, E) 
\ar[rr]^{f_{*}} 
& &\Ext^1(f_{*}E, f_{*}E)=T_{[f_{*}E]}M^{H}(w). 
}
\]
Note that since $E \in M^0(v) \cap M^1(v)$, 
we have $\dR f_{*}E=f_{*}E$ and 
$\dL f^* \dR f_{*}E \cong f^*f_{*}E$ 
(see Lemma \ref{pull back in per}). 
Applying $\Hom(-, E)$ to the second row 
of the diagram (\ref{ex seq}), we get 
\[
0=\Hom(K, E) \to \Ext^1(E, E) \hookrightarrow 
\Ext^1(f^{*}f_{*}E, E) \cong \Ext^1(f_{*}E, f_{*}E). 
\]
Note that $\Hom(K, E)=0$ follows from 
the $1$-stability of $E$. 
By \cite[Lemma 1.21]{huy06}, 
the diagram 
\[
\xymatrix{
&\Ext^1(E, E) \ar[d] \ar[rd]^{f_{*}} & \\ 
&\Ext^1(f^*f_{*}E, E) \ar[r] & \Ext^1(f_{*}E, f_{*}E)
}
\]
is commutative. 
Hence we conclude that 
the tangent map $f_{*}$ is injective. 

(2) For $G=\xi(E) \in \xi(M^0(v))$, we 
associate an element 
$\Phi(G)$ of 
$\xi^{-}_{0}(M^0(v))$ as follows: 
we have a commutative diagram 
\[
\xymatrix{
&0 \ar[r] &T' \ar[r] \ar@{.>}[d] & f^{*}f_{*}E \ar[r] \ar@{=}[d] & E \ar[r] \ar@{.>>}[d] &0 \\ 
&0 \ar[r] &T \ar[r] &f^{*}f_{*}E \ar[r] &F \ar[r] &0 
}
\]
in $\Coh(X)$, 
where $T'$ is the kernel of the adjoint map 
and the bottom sequence is the decomposition 
with respect to the torsion pair 
$(\mathcal{T}_{D}, \mathcal{F}_{D})$. 
Since $f^*f_{*}E$ is $0$-stable, 
we have 
$T \in \mathcal{C}^0$ and 
$F$ is $1$-stable 
by Proposition \ref{0-st to 1-st}. 
Note that since $\Hom(\mathcal{C}^0, F)=0$, 
we have a surjective map $E \to F \to 0$. 
Furthermore, by the construction, its kernel 
$\Ker(E \to F) \in \mathcal{C}^0$. 
In other words, 
\begin{equation}
\label{def of Phi}
\Phi(G):=F=\xi_{0}^{-}(E) \in \xi_{0}^{-}(M^0(v)). 
\end{equation}
Note that the definition of $\Phi$ is independent of 
the choice of $E \in M^0(v)$ with $f_{*}E=G$, i.e. 
the map $\Phi$ is well-defined. 
Indeed, $F$ only depends on $f^{*}f_{*}E=f^{*}G$. 

\begin{claim}
The map $\Phi$ is injective. 
\end{claim}

\begin{proof}
Let $G, G' \in \xi(M^0(v))$ with 
$\Phi(G)=\Phi(G')$. 
Then we have 
$G=f_{*}\Phi(G)=f_{*}\Phi(G')=G'$. 
\end{proof}

\begin{claim}
The map $\Phi$ is surjective. 
\end{claim}

\begin{proof}
Let $F=\xi^{-}_{0}(E) \in \xi_{0}^{-}(M^0(v))$. 
Then the equation (\ref{def of Phi}) shows that 
$F=\Phi(\xi(E))$. 
\end{proof}

(3) The assertion follows from the equation (\ref{def of Phi}). 

\end{proof}

\begin{prop}
\label{xi vs xi+}
\begin{enumerate}
\item We can identify $\xi^{+}_{0}(M^1(v))$ 
with $\xi^{+}(M^1(v))$. 

\item Via the identification (1), 
we have 
$\xi^{+}=\xi^{+}_{0}$. 
\end{enumerate}
\end{prop}

\begin{proof}
The proof is similar to that of Proposition \ref{xi vs xi0}. 
Hence we just give an outline of the proof. 
We define the map 
$\Psi \colon \xi^+(M^1(v)) \to \xi^+_{0}(M^1(v))$ 
as follows: 
Let $F=\xi^{+}(E) \in \xi^{+}(M^1(v))$. 
Then we have $\phi \colon f_{*}E \cong F$ 
by definition. 
Take an element $\alpha \in \Hom(E, f^{*}F(D))$ 
corresponding to $\phi$ via the isomorphism 
$\Hom(f_{*}E, F) \cong \Hom(E, f^{*}F(D))$. 
Let $C:=\Cone(\alpha)$. 
Then since $f_{*}(\alpha)=\phi$, 
we have $C \in \mathcal{C}$. 
Furthermore, since $E$ is $1$-stable, 
we must have $\mathcal{H}^{-1}(C)=0$, 
i.e. $C$ is a sheaf. 
Hence we have the exact sequence 
\begin{equation}
\label{sub}
\xymatrix{
&0 \ar[r] &E \ar[r]^{\alpha} &f^{*}F(D) \ar[r] &C \ar[r] &0
}
\end{equation}
in $\Coh(X)$. 
Consider the following commutative diagram 
\[
\xymatrix{
& &0 \ar[d] & 0 \ar[d] & & \\
&0 \ar[r] &G' \ar[r] \ar[d] & E \ar[r] \ar^{\alpha}[d] & T' \ar[r] \ar[d] &0 \\ 
&0 \ar[r] &G \ar[r] &f^{*}F(D) \ar[r] &T \ar[r] &0 
}
\]
with $G', G \in \mathcal{T}$ 
and $T', T \in \mathcal{F}$. 
Using the fact that 
$C=\Coker(\alpha) \in \mathcal{C}^0$, 
we can show that $G' \cong G$. 
Now we define as 
$\Psi(F):=G=G'=\xi^+_{0}(E) \in \xi^+_{0}(M^1(v))$. 
As in Proposition \ref{xi vs xi0}, 
we can prove that $\Psi$ is bijection. 
Furthermore, the second assertion follows from 
the definition of the map $\Psi$. 
\end{proof}

As a summary of 
Proposition \ref{set}, 
Proposition \ref{xi vs xi0}, and 
Proposition \ref{xi vs xi+}, 
we get: 

\begin{prop}
\label{scheme}
There exists a diagram of projective schemes 
\[
\xymatrix{
&\cdots & M^m(v) \ar[ld] \ar[rd]^{\xi_{m}^{-}} & 
&M^{m+1}(v) \ar[ld]_{\xi_{m}^+} \ar[rd] & &\cdots \\
& & &M^{m, m+1}(v) & & & 
}
\]
connecting the moduli spaces of $m$-stable sheaves, 
where the morphisms $\xi_{m}^{\pm}$ are defined as 
$\xi_{m}^{\pm}:=f_{*}(- \otimes \mathcal{O}(-mD))$. 
Moreover, we can also describe 
the morphisms $\xi^{+}_{m}, \xi^{-}_{m}$ in terms of 
the decompositions with respect to the torsion pairs 
$\left(\mathcal{T} \otimes \mathcal{O}(mD), 
\mathcal{F} \otimes \mathcal{O}(mD)
\right)$, 
$\left(\mathcal{T}_{D} \otimes \mathcal{O}(mD), 
\mathcal{F}_{D} \otimes \mathcal{O}(mD)
\right)$, 
respectively. 
\end{prop}

We end with the following proposition 
describing the fibers of  $\xi$ and $\xi^{+}$. 
We can think them as the scheme structures on 
$\PSub(F, \beta)$ and 
$\widetilde{\PQuot}(F, \beta)$ 
defined in Definition \ref{pquot}. 

\begin{prop}
\label{fiber}
Let $F \in M^{H}(w)$ 
be a $H$-stable sheaf 
with Chern character $\ch(F)=v'$. 
Then the following statements hold. 
\begin{enumerate}
\item We have 
$
\xi^{-1}(F)=\Quot(f^{*}F, v). 
$

\item We have 
$
(\xi^{+})^{-1}(F)=\Sub(f^{*}F(D), v). 
$
\end{enumerate}
\end{prop}

\begin{proof}
(1) Let $E \in \xi^{-1}(F)$. 
Then we have an exact sequence 
\[
0 \to T \to f^{*}F=f^{*}f_{*}E \to E \to 0
\]
in $\Coh(X)$, i.e. 
$E \in \Quot(f^{*}F, v)$. 

For the converse, let $E \in \Quot(f^{*}F, v)$. 
Then we have 
\[
0 \to K \to f^{*}F \to E \to 0. 
\]
Note that $E \in \Per(X/Y)$. 
Since $K$ is torsion and $F$ is torsion free, 
we have $f_{*}K=0$ and 
$F \hookrightarrow f_{*}E$ is injective. 
Moreover, since $\ch(F)=\ch(f_{*}E)$, we have 
$F \cong f_{*}E$. 
Hence $E \in \xi^{-1}(F)$. 
Noting that 
$\Hom(f^{*}F, E) \cong \Hom(F, f_{*}E)=\mathbb{C}$, 
we conclude that 
$\xi^{-1}(F)=\Quot(f^{*}F, v)$. 

(2) Let $E \in (\xi^{+})^{-1}(F)$. 
Then by the proof of Proposition \ref{xi vs xi+}, 
we have the exact sequence as in (\ref{sub}): 
\[
0 \to E \to f^{*}F(D) \to C \to 0. 
\]
In other words, $E \in \Sub(f^{*}F(D), v)$. 

For the converse, let $E \in \Sub(f^{*}F(D), v)$.  
Then we have 
\[
0 \to E \to f^{*}F(D) \to C \to 0. 
\]

First we claim that $C \in \mathcal{C}^0$. 
Since $f^{*}F \in \Per(X/Y)$, 
we have $C(-D) \in \Per(X/Y)$ 
by Lemma \ref{property tor pair}. 
Hence by Lemma \ref{r1 per}, we have $\dR^1f_{*}C=0$. 
Furthermore, we have $\ch(C) \in \mathcal{S}$. 
Hence 
$\ch(f_{*}C)=\ch(\dR f_{*}C)=0$ 
and so we also have $f_{*}C=0$. 

Next we show that $E(-D) \in \Per(X/Y)$. 
To show this, it is enough to show that 
for $y \in C$, we have 
$\Hom(E(-D), \mathcal{O}_{L_{y}}(-1))=0$. 
Write $C(-D)=f^{*}M$. 
Applying $\Hom(-, \mathcal{O}_{L_{y}}(-1))$ 
to the exact sequence 
\[
0 \to E(-D) \to f^{*}F \to f^{*}M \to 0, 
\]
we have 
\[
0=\Hom(f^{*}F, \mathcal{O}_{L_{y}}(-1)) \to 
\Hom(E(-D), \mathcal{O}_{L_{y}}(-1)) \to 
\Hom(f^{*}M, \mathcal{O}_{L_{y}}(-1)[1]). 
\]

Hence it is enough to show that 
$\Hom(f^{*}M, \mathcal{O}_{L_{y}}(-1)[1])=0$. 
From the exact triangle 
\[
\dL^{-1}f^{*}M[1] \to \dL f^{*}M \to f^{*}M, 
\]
we have 
\begin{align*}
0=\Hom\left(
 \dL^{-1}f^{*}M[1], \mathcal{O}_{L_{y}}(-1)
 \right) 
&\to 
\Hom\left(
f^{*}M, 
\mathcal{O}_{L_{y}}(-1)[1]
 \right) \\
&\to 
\Hom\left(
 \dL f^{*}M, \mathcal{O}_{L_{y}}(-1)[1]
 \right)& \\
&\cong \Hom\left(
M, \dR f_{*}\mathcal{O}_{L_{y}}(-1)[1]
\right) \\ 
&=0. 
\end{align*}

As a conclusion, we have $E(-D) \in \Per(X/Y)$. 
Furthermore, since $f^{*}F(D) \in \mathcal{F}_{D}$, 
we also have $E \in \mathcal{F}_{D}$. 
Hence by Lemma \ref{criterion for tor free}, 
$f_{*}E(-D)$ is torsion free. 
Since $f_{*}E \cong F$ is $\mu$-stable, 
we conclude that 
$f_{*}E(-D)$ is also $\mu$-stable, i.e. 
$E \in (\xi^{+})^{-1}(F)$. 
As before, noting that 
$\Hom(E, f^{*}F(D)) \cong \Hom(f_{*}E, F)=\mathbb{C}$, 
we have 
$(\xi^{+})^{-1}(F)=\Sub(f^{*}F(D), v)$. 
\end{proof}


\section{Hilbert scheme of two points}
\label{Hilb}
In this section, we study the birational geometry 
of Hilbert scheme of two points using the 
flip-like diagram (\ref{zig-zag}) 
constructed in the previous sections. 
In the followings, we assume that 
$H^1(\mathcal{O}_{Y})=0$ and 
let $w:=(1, 0, \cdots, 0, -2) \in H^{*}(Y; \mathbb{Q})$, 
$v:=f^{*}w \in H^{*}(X; \mathbb{Q})$. 
Then we have 
$M^{H}(w)=\Hilb^2(Y)$, 
$M^{f^{*}H-\epsilon D}(v)=\Hilb^2(X).$ 
We will use that following notations: 
\begin{itemize}
\item For a $0$-dimensional closed subscheme $Z \subset X$ 
of length $2$, we denote its ideal sheaf as 
$I_{Z} \in \Hilb^2(X)$. 

\item $\Hilb^2(D/C) \subset \Hilb^2(X)$ denotes 
the relative Hilbert scheme which parametrizes 
$I_{Z} \in \Hilb^2(X)$ such that $Z$ is 
scheme-theoretically contained in a fiber 
of $\pi \colon D \to C$, i.e. 
there exists $y \in C$ such that 
$Z \subset L_{y}$. 

\item $\mathcal{I} \in \Coh(\Hilb^2(X) \times X)$ 
denotes the universal ideal sheaf on 
$\Hilb^2(X)$.  
\end{itemize}

We start with the following lemma: 

\begin{lem}
\label{stability of ideal sheaf}
Let $I_{Z} \in \Hilb^2(X)$ 
be an ideal sheaf of a length two 
closed subscheme $Z \subset X$. 
Then the following holds: 
\begin{enumerate}
\item If $Z \cap D= \emptyset$, then 
$I_{Z}$ is $0$-stable. 

\item If $Z \cap D \neq \emptyset$ and 
$I_{Z} \notin \Hilb^2(D/C)$, then 
$I_{Z}$ is $1$-stable but not $0$-stable. 

\item If $I_{Z} \in \Hilb^2(D/C)$, then 
$I_{Z}$ is $2$-stable but not $1$-stable. 
\end{enumerate}
\end{lem}

\begin{proof}
By the proof of Proposition \ref{moduli on X}, 
it is enough to find the smallest 
$m \in \mathbb{Z}_{\geq 0}$ such that 
$I_{Z}(-mD) \in \Per(X/Y)$, i.e. 
\[
\Hom\left(
I_{Z}, \mathcal{O}_{L_{y}}(-m-1)
\right)
=\Hom\left(
I_{Z}(-mD), \mathcal{O}_{L_{y}}(-1)
\right)
=0 
\]
for all $y \in C$. 

Restricting the exact sequence 
\[
0 \to I_{Z} \to \mathcal{O}_{X} \to \mathcal{O}_{Z} \to 0
\]
to $L_{y}$, we get 
\[
0 \to T \to I_{Z}|_{L_{y}} \to \mathcal{O}_{L_{y}} 
\to \mathcal{O}_{Z \cap L_{y}} \to 0,  
\]
where $T \in \Coh(L_{y})$ is some torsion sheaf. 
Hence we get 
\[
0 \to T \to I_{Z}|_{L_{y}} \to \mathcal{O}_{L_{y}}(-l) \to 0,  
\]
where $l \in \mathbb{Z}$ is the length of $Z \cap L_{y}$. 
We conclude that  
\begin{align*}
\Hom\left(
I_{Z}, \mathcal{O}_{L_{y}}(-m-1)
\right)
&=\Hom\left(
I_{Z}|_{L_{y}}, \mathcal{O}_{L_{y}}(-m-1)
\right) \\ 
&=\Hom\left(
\mathcal{O}_{L_{y}}(-l), \mathcal{O}_{L_{y}}(-m-1)
\right). 
\end{align*}

Hence we get the result. 
\end{proof}

\begin{cor}
\label{exc1+}
The exceptional locus of $\xi^+_{1}$ is 
$\Exc(\xi^+_{1})=\Hilb^2(D/C)$. 
\end{cor}
\begin{proof}
The assertion directly follows from 
Lemma \ref{stability of ideal sheaf}. 
\end{proof}

By Lemma \ref{stability of ideal sheaf}, 
we have the diagram: 

\begin{equation}
\xymatrix{
& &\ar[ld]^{\xi_{0}=f_{*}} \widetilde{M}^1(v) \ar[rd]_{\xi_{1}^{-}} & &\Hilb^2(X) \ar[ld]^{\xi_{1}^{+}} \subset M^2(v) \\
&\Hilb^2(Y) & &M^{1, 2}(v), & 
}
\end{equation}
where $\widetilde{M}^1(v)$ 
denotes the normalization of 
the connected component of 
$M^1(v)$ containing 
$\Hilb^2(X) \setminus \Hilb^2(D/C)$. 
In fact, we can actually show that the connectedness of 
$M^1(v)$ by using Proposition \ref{fiber}. 
However, we omit the proof here. 

In the following subsections, we will study 
the properties of these morphisms in details. 


\subsection{The diagram $\xi^{\pm}_{1}$ is a flip}
In this subsection, we will show that 
the diagram $\xi_{1}^{\pm}$ is a flip. 
The main technical tool used here is so-called 
elementary transformation. 
First we observe the following. 

\begin{lem}
\label{relative hilb}
Let $Z \subset L_{y}$ be a 
length $2$ closed subscheme ($y \in C$). 
Then the following statements hold: 
\begin{enumerate}
\item We have an exact sequence 
\begin{equation}
\label{decomp ideal}
0 \to I_{L_{y}} \to I_{Z} \to \mathcal{O}_{L_{y}}(-2) \to 0.  
\end{equation}

\item We have $I_{L_{y}} \in \Per(X/Y)$. 
\end{enumerate}
In particular, the sequence (\ref{decomp ideal}) 
is the decomposition with respect to the torsion pair 
$\left(
\mathcal{T} \otimes \mathcal{O}(D), 
\mathcal{F} \otimes \mathcal{O}(D) 
\right)$. 
\end{lem}

\begin{proof}
(1) Since $Z \subset L_{y}$ is length $2$, 
we have 
\[
0 \to \mathcal{O}_{L_{y}}(-2) \to \mathcal{O}_{L_{y}} \to \mathcal{O}_{Z} \to 0. 
\]
Hence we get the following 
commutative diagram: 
\[
\xymatrix{
& & & &0 \ar[d] & \\ 
& &0 \ar[d] & & \mathcal{O}_{L_{y}}(-2) \ar[d] & \\ 
&0 \ar[r] &I_{L_{y}} \ar[r] \ar[d] &\mathcal{O}_{X} \ar[r] \ar@{=}[d] &\mathcal{O}_{L_{y}} \ar[r] \ar[d] &0 \\ 
&0 \ar[r] &I_{Z} \ar[r] \ar[d]&\mathcal{O}_{X} \ar[r] &\mathcal{O}_{Z} \ar[r] \ar[d] &0 \\ 
& &T \ar[d] & &0 \\ 
& &0, 
}
\]
which implies $T=\mathcal{O}_{L_{y}}(-2)$ 
as required. 

(2) Pulling back the exact sequence 
\[
0 \to I_{y} \to \mathcal{O}_{Y} \to \mathcal{O}_{y} \to 0, 
\]
we get 
\[
0 \to \mathcal{O}_{L_{y}}(-1) \to f^{*}I_{y} \to \mathcal{O}_{X} \to \mathcal{O}_{L_{y}} \to 0. 
\]
It gives two short exact sequences 
\begin{align*}
&0 \to \mathcal{O}_{L_{y}}(-1) \to f^{*}I_{y} \to I_{L_{y}} \to 0, \\
&0 \to I_{L_{y}} \to \mathcal{O}_{X} \to \mathcal{O}_{L_{y}} \to 0. 
\end{align*}

By Lemma \ref{pull back in per}, 
we have 
$f^{*}I_{y} \in \Per(X/Y)$. 
Since there exists a surjection 
$f^{*}I_{y} \to I_{L_{y}}$, 
we also have $I_{L_{y}} \in \Per(X/Y)$ 
by Lemma \ref{property tor pair}. 

For the last assertion, we first note that 
$\mathcal{O}_{L_{y}}(-2) \otimes \mathcal{O}(-D)
\cong 
\mathcal{O}_{L_{y}}(-1)
\subset \mathcal{F}$. 
Furthermore, by (2) and Corollary \ref{cor of criterion}, 
we have $I_{L_{y}}(-D) \in \mathcal{T}$. 
Hence the sequence (\ref{decomp ideal}) 
is the decomposition with respect to the torsion pair 
$\left(
\mathcal{T} \otimes \mathcal{O}(D), 
\mathcal{F} \otimes \mathcal{O}(D) 
\right)$. 

\end{proof}

\begin{cor}
\label{fiber of xi1}
\begin{enumerate}
\item The restriction of $\xi^{+}_{1}$ to $\Hilb^2(D/C)$ 
is given by 
\[
\xi^{+}_{1} \colon \Hilb^2(D/C) \to C, \quad 
I_{Z} \mapsto I_{L_{y}}.  
\]

\item For every closed point $y$ of $C$, 
the fiber of $\xi^{+}_{1}$ over $I_{L_{y}}$ is 
given by 
\[
(\xi^{+}_{1})^{-1}(I_{L_{y}})=
\Hilb^2(L_{y}) \cong \mathbb{P}^2. 
\]

\item For every closed point $y$ of $C$, 
the fiber of $\xi^{-}_{1}$ over $I_{L_{y}}$ is 
given by 
\[
(\xi_{1}^-)^{-1}(I_{L_{y}})
=\mathbb{P}
\Ext^1\left(
I_{L_{y}}, \mathcal{O}_{L_{y}}(-2)
\right)^{\vee}. 
\]
\end{enumerate}
\end{cor}
\begin{proof}
(1) Recall from Proposition \ref{set} that 
the morphism $\xi_{1}^+$ is defined by using 
the decomposition with respect to the torsion pair 
$\left(
\mathcal{T} \otimes \mathcal{O}(D), 
\mathcal{F} \otimes \mathcal{O}(D) 
\right)$. 
By Lemma \ref{relative hilb}, 
we have $\xi_{1}^+(I_{Z})=I_{L_{y}}$ 
as required. 

(2) By Corollary \ref{exc1+} and (1), 
for $I_{Z} \in \Hilb^2(X)$, 
we have $\xi_{1}^+(I_{Z})=I_{L_{y}}$ if and only if 
$Z$ is the closed subscheme of $L_{y}$, 
i.e. $I_{Z} \in \Hilb^2(L_{y})$. 
This proves the assertion. 

(3) For a $1$-stable sheaf $E \in {M}^1(v)$, 
we have $E \in (\xi_{1}^-)^{-1}(I_{L_{y}})$ 
if and only if 
there exists an element 
$T \in \mathcal{C}^0 \otimes \mathcal{O}(D)$ 
such that 
$E$ fits into a non-trivial exact sequence 
\[
0 \to T \to E \to I_{L_{y}} \to 0 
\]
by the construction of the morphism 
$\xi_{1}^-$. 
Note that we have 
$\ch(T)=\ch(E)-\ch(I_{L_{y}})
=\ch(I_{Z})-\ch(I_{L_{y}})
=\ch(\mathcal{O}_{L_{y}}(-2))$, 
where $I_{Z} \in \Hilb^2(X)$. 
Hence the only possibility is $T=\mathcal{O}_{L_{y}}(-2)$. 
We conclude that 
$(\xi_{1}^-)^{-1}(I_{L_{y}})
=\mathbb{P}
\Ext^1\left(
I_{L_{y}}, \mathcal{O}_{L_{y}}(-2)
\right)^{\vee}
$. 

\end{proof}

The following lemma determines the dimension of 
$\mathbb{P}
\Ext^1\left(
I_{L_{y}}, \mathcal{O}_{L_{y}}(-2)
\right)^{\vee}
$: 

\begin{lem}
\label{dim of fiber}
For every $y \in C$, 
we have 
$\ext^1(I_{L_{y}}, \mathcal{O}_{L_{y}}(-2))
=n$. 
\end{lem}
\begin{proof}
Applying 
$\Hom\left(
-, \mathcal{O}_{L_{y}}(-2)
\right)$ 
to the standard exact sequence 
\[
0 \to I_{L_{y}} \to \mathcal{O}_{X} \to \mathcal{O}_{L_{y}} \to 0, 
\]
the claim is reduced to compute 
$\ext^i(\mathcal{O}_{L_{y}}, \mathcal{O}_{L_{y}}(-2))$ 
for $i=1, 2$. 
Using the local-to-global spectral sequence 
\[
E_{2}^{p, q}
=H^{p}\left(
X, \mathcal{E}xt^q\left(
\mathcal{O}_{L_{y}}, \mathcal{O}_{L_{y}}(-2)
\right)
\right)
\Rightarrow 
\Ext^{p+q}(\mathcal{O}_{L_{y}}, \mathcal{O}_{L_{y}}(-2)) 
\]
and the isomorphism 
\begin{align*}
\mathcal{E}xt^q(\mathcal{O}_{L_{y}}, \mathcal{O}_{L_{y}}(-2))
&\cong 
\bigwedge^q \mathcal{N}_{L_{y}/X} \otimes \mathcal{O}_{L_{y}}(-2) \\
&=\bigwedge^q\left(
\mathcal{O}_{L_{y}}^{\oplus n-2} 
\oplus 
\mathcal{O}_{L_{y}}(-1)
\right)
\otimes 
\mathcal{O}_{L_{y}}(-2) \\
&\cong 
\mathcal{O}_{L_{y}}(-2)^{\oplus \binom{n-2}{q}} 
\oplus 
\mathcal{O}_{L_{y}}(-3)^{\oplus \binom{n-2}{q-1}} 
\end{align*}
(cf. \cite[Corollary 11.2]{huy06}), 
we have 
\[
\ext^i\left(
\mathcal{O}_{L_{y}}, 
\mathcal{O}_{L_{y}}(-2)
\right)
=
\begin{cases}
1 & (i=1) \\
(n-2)+2=n & (i=2). 
\end{cases}
\]
Using them, we get the result. 
\end{proof}

We next compute the normal bundle 
$\mathcal{N}_{\Hilb^2(D/C)/\Hilb^2(X)}$. 
To do that, let us define an embedding 
\[
D \to C\times X, \quad x \mapsto (f(x), x) 
\]
and denote the ideal sheaf of $D$ in $C \times X$ 
as $I_{D/ C\times X}$. 
We also use the following notations for projections: 

\[
\xymatrix{
& & \Hilb^2(X) \times X \ar[ld]^{p_{H}} \ar[rd]_{p_{X}} & 
& & C \times X \ar[ld]^{q_{C}} \ar[rd]_{q_{X}} & \\ 
&\Hilb^2(X) & & X 
&C & & X \\
}
\]

\begin{defin}
We define sheaves 
$\mathcal{E}_{\pm} \in \Coh(C)$ 
and 
$\mathcal{E} \in \Coh(\Hilb^2(X))$ 
as follows: 
\begin{align*}
&\mathcal{E}_{+}:=
\mathcal{E}xt^1_{q_{C}}\left(
\mathcal{O}_{D}(2D), 
I_{D/C \times X}
\right), \quad 
\mathcal{E}_{-}:=
\mathcal{E}xt^1_{q_{C}}\left(
I_{D/ C \times X}, 
\mathcal{O}_{D}(2D)
\right), \\
&\mathcal{E}:=
\mathcal{E}xt^1_{p_{H}}\left(
\mathcal{I}, 
\mathcal{I}
\right). 
\end{align*}
\end{defin}

We recall the following version of semicontinuity theorem: 

\begin{thm}[{\cite[Satz 3]{bps80}}]
\label{semicontinuity}
Let $g \colon M \to N$ 
be a flat morphism between 
smooth projective varieties. 
Let $E, F \in \Coh(M)$ 
be flat sheaves over $N$. 
Fix an integer 
$i \in \mathbb{Z}_{\geq 0}$. 
Assume that for every $p \in N$, 
$\ext^i(E_{p}, F_{p})$ is constant, 
where $E_{p}, F_{p} \in \Coh(M_{p})$ 
is the restriction of $E, F$ to 
the fiber $M_{p}:=g^{-1}(p)$. 
Then the sheaf 
$\mathcal{E}xt^i_{g}(E, F)$ 
is locally free. 
Furthermore, for $q=i-1, i$, 
$\mathcal{E}xt^q_{g}(E, F)$ commutes with the base change, i,e, 
for every $p \in N$, we have an isomorphism 
\[
\mathcal{E}xt^{q}_{g}(E, F)|_{\{p\}} 
\cong \Ext^q(E_{p}, F_{p}). 
\]
\end{thm}

\begin{cor}
\begin{enumerate}
\item The sheaves 
$\mathcal{E}_{\pm}, \mathcal{E}$ 
are locally free. 

\item We have 
$\Hilb^2(D/C) \cong \mathbb{P}(\mathcal{E_{+}}^{\vee})$. 

\item The Hilbert scheme $\Hilb^2(X)$ 
is smooth of dimension $2n$ and 
its tangent bundle is given as 
$\mathcal{T}_{\Hilb^2(X)} 
\cong \mathcal{E}$. 
\end{enumerate}
\end{cor}

\begin{proof}
First of all, 
the smoothness of $\Hilb^2(X)$ is well known: 
the Hilbert-Chow morphism 
$\Hilb^2(X) \to \Sym^2(X)$ coincides with 
$\Bl_{\Delta_{X}}\Sym^2(X)$, i.e. 
$\Hilb^2(X) \cong \Bl_{\Delta_{X}}\Sym^2(X)$. 
Moreover, $\Sym^2(X)$ has only 
$\frac{1}{2}(1, \cdots, 1)$-type singularity 
along the diagonal $\Delta_{X} \subset \Sym^2(X)$. 
Hence $\Hilb^2(X)$ is smooth. 

(1) Since $D$ is flat over $C$, 
$I_{D/C \times D}$ and $\mathcal{O}_{D}(2D)$ 
are flat over $C$. 
Hence by Lemma \ref{dim of fiber} and 
Theorem \ref{semicontinuity}, 
the assertion follows. 

(2) We have the universal extension sheaf 
$\mathcal{F} \in 
\Coh\left(
\mathbb{P}(\mathcal{E}_{+}^{\vee}) 
\times X 
\right)$ 
which fits into the exact sequence 
\begin{equation}
\label{univ ext}
0 \to 
(\pi^{+}_{X})^{*} I_{D/C \times X} \otimes p^{*} \mathcal{O}_{\pi^{+}}(1) 
\to \mathcal{F} \to (\pi^{+}_{X})^{*} \mathcal{O}_{D}(2D) \to 0 
\end{equation}
(cf. \cite[Example 2.1.12]{hl97}). 
Here $\pi^{+} \colon \mathbb{P}(\mathcal{E}_{+}^{\vee}) \to C$ 
is the structure morphism of the projective space bundle, 
$\pi^{+}_{X} \colon 
\mathbb{P}(\mathcal{E}_{+}^{\vee}) \times X 
\to C \times X$ 
is the base change morphism, 
and $p \colon \mathbb{P}(\mathcal{E}_{+}^{\vee}) \times X 
\to \mathbb{P}(\mathcal{E}_{+}^{\vee})$ is the projection. 
The sheaf $\mathcal{F}$ parametrizes all the extensions 
\[
0 \to I_{L_{y}} \to F \to \mathcal{O}_{L_{y}}(-2) \to 0. 
\]

By Lemma \ref{relative hilb}, we have 
$F \cong I_{Z} \in \Hilb^2(D/C)$. 
Hence by the universality of the Hilbert scheme, 
we get 
$\Hilb^2(D/C) \cong \mathbb{P}(\mathcal{E}_{+}^{\vee})$ 
as required. 

(3) Consider the Kodaira-Spencer map 
$KS \colon \mathcal{T}_{\Hilb^2(X)} 
\to \mathcal{E}$. 
Since $\mathcal{E}$ commutes with the base change, 
$KS$ restricts to an isomorphism 
$KS_{p} \colon T_{p}\Hilb^2(X) \to \Ext^1(I_{Z}, I_{Z})$
for each $p=[I_{Z}] \in \Hilb^2(X)$. 
Hence $KS$ is surjective morphism 
between locally free sheaves of the same rank. 
We conclude that 
$\mathcal{T}_{\Hilb^2(X)} 
\cong \mathcal{E}$. 
\end{proof}

Now we can compute the normal bundle: 

\begin{lem}[cf. {\cite[Proposition 3.7]{fq95}}]
\label{normal bundle}
We have an isomorphism 
\[
\mathcal{N}_{\Hilb^2(D/C)/\Hilb^2(X)}
\cong (\pi^{+})^{*}\mathcal{E_{-}} \otimes \mathcal{O}_{\pi^{+}}(-1). 
\]
\end{lem}

\begin{proof}
First we construct a morphism 
$\mathcal{E}|_{\Hilb^2(D/C)} 
\to (\pi^{+})^{*}\mathcal{E}_{-} 
\otimes \mathcal{O}_{\pi^{+}}(-1)$. 
Applying 
$\dR\pi^{+}_{*}\dR\mathcal{H}om(\mathcal{I}|_{\Hilb^2(D/C)}, -)$ 
and 
$\dR\pi^{+}_{*}\dR\mathcal{H}om\left(
-, (\pi^{+}_{X})^{*}\mathcal{O}_{D}(2D)
\right)$ 
to the exact sequence (\ref{univ ext}) 
and taking its cohomology, 
we get 
\begin{equation} 
\delta_{1} \colon 
\mathcal{E}|_{\Hilb^2(D/C)} \to
\mathcal{E}xt^1_{\pi^{+}}\left(
\mathcal{I}|_{\Hilb^2(D/C)}, 
(\pi^{+}_{X})^{*}\mathcal{O}_{D}(2D)
\right) 
\end{equation} 
and 
\begin{align} 
\delta_{2} \colon 
\mathcal{E}xt^1_{\pi^{+}}
&\left(
\mathcal{I}|_{\Hilb^2(D/C)}, 
(\pi^{+}_{X})^{*}\mathcal{O}_{D}(2D)
\right) 
\to \\
&\mathcal{E}xt^1_{\pi^{+}}\left(
(\pi^{+}_{X})^{*} I_{D/C \times X} \otimes p^{*} \mathcal{O}_{\pi^{+}}(1), 
(\pi^{+}_{X})^{*}\mathcal{O}_{D}(2D)
\right). \notag 
\end{align} 

Note that we used 
$\mathcal{F}\cong \mathcal{I}|_{\Hilb^2(D/C)}$ 
above. 
Straightforward computation shows that 
\[
\mathcal{E}xt^1_{\pi^{+}}\left(
(\pi^{+}_{X})^{*} I_{D/C \times X} \otimes p^{*} \mathcal{O}_{\pi^{+}}(1), 
(\pi^{+}_{X})^{*}\mathcal{O}_{D}(2D)
\right) 
\cong (\pi^{+})^{*}\mathcal{E}^{-} \otimes \mathcal{O}_{\pi^{+}}(-1). 
\]

Hence we get the morphism 
\[
\delta:=\delta_{2} \circ \delta_{1} \colon 
\mathcal{E}|_{\Hilb^2(D/C)} 
\to (\pi^{+})^{*}\mathcal{E}_{-} 
\otimes \mathcal{O}_{\pi^{+}}(-1). 
\]

We will prove that 
$\delta$ is surjective and 
$\ker \delta 
\cong \mathcal{T}_{\Hilb^2(D/C)}$. 
To show that, it is enough to show the following: 
For every $p=[I_{Z}] \in \Hilb^2(D/C)$ with 
$Z \subset L_{y}$, 
\begin{enumerate}
\item the restriction 
$
\delta_{p} \colon 
T_{p}\Hilb^2(X)|_{\Hilb^2(D/C)} 
\to \Ext^1(I_{L_{y}}, \mathcal{O}_{L_{y}}(-2))
$
is surjective, 

\item $\ker(\delta_{p})=T_{p}\Hilb^2(D/C)$. 
\end{enumerate}

In (1), we can actually show that 
both $\delta_{1, p}$ and $\delta_{2, p}$ 
are surjective 
by using the spectral sequence argument 
as in Lemma \ref{dim of fiber}. 

For (2), it is now enough to show that 
$T_{p}\Hilb^2(D/C) \subset \ker(\delta_{p})$ 
since they are vector spaces of the same dimension. 
The argument is exactly same as 
\cite[Proposition 3.7]{fq95} and  
hence we omit the detail. 
\end{proof}

The following result directly follows from 
Lemma \ref{normal bundle}: 

\begin{cor}
\begin{enumerate}
\item When $n=2$, 
$\xi^{+}_{1}$ is a flipping contraction. 

\item When $n=3$, 
$\xi^{+}_{1}$ is a flopping contraction. 

\item When $n \geq 4$, 
$\xi^{+}_{1}$ is an anti-flipping contraction. 
\end{enumerate}
\end{cor}

Let $\mu \colon 
M:=\Bl_{\Hilb^2(D/C)}\Hilb^2(X) 
\to \Hilb^2(X)$, 
$E \subset M$ be the $\mu$-exceptional divisor, 
$\nu:=\mu|_{E} \colon D \to \Hilb^2(D/C)$. 
Note that by Lemma \ref{normal bundle}, 
we know that 
$E \cong 
\mathbb{P}(\mathcal{E}_{-}^{\vee}) 
\times_{C} 
\mathbb{P}(\mathcal{E}_{+}^{\vee}) 
$. 
By the elementary transformation, 
we will construct the family 
$\mathcal{G}$ 
of $1$-stable 
sheaves on $M$, which gives us the morphism 
$M \to \widetilde{M}^1(v)$. 

Pulling back the exact sequence (\ref{univ ext}), 
we get 
\[
0 \to \nu^{*}_{X}(\pi^{+}_{X})^{*}\left(
I_{D/C \times X} 
\otimes (\mathcal{O}_{\pi^{+}}(1)) 
\right) 
\to \mu_{X}^{*}\mathcal{I}|_{E \times X} 
\to \nu_{X}^{*}(\pi^{+})^*\mathcal{O}_{D}(2D) 
\to 0 
\]
in $\Coh(M \times X)$. 
Now define 
\[
\mathcal{G}:=
\ker\left(
\mu_{X}^*\mathcal{I} \to 
\mu_{X}^{*}\mathcal{I}|_{E \times X} 
\to \nu_{X}^{*}(\pi^{+})^*\mathcal{O}_{D}(2D) 
\right). 
\]

\begin{lem}
\label{elementary transform}
\begin{enumerate}
\item The sheaf 
$\mathcal{G}$ is a flat family of 
$1$-stable sheaves and hence 
defines the morphism 
$\sigma \colon M \to \widetilde{M}^1(v)$. 

\item The restriction 
$\sigma|_{E} \colon E \to \sigma(E)$ 
coincides with the projection morphism 
$\mathbb{P}(\mathcal{E}_{-}^{\vee}) 
\times_{C} 
\mathbb{P}(\mathcal{E}_{+}^{\vee}) 
\to \mathbb{P}(\mathcal{E}_{-}^{\vee})$. 
\end{enumerate}
\end{lem}
\begin{proof}
(1) We get the following commutative diagram: 

\[
\xymatrix{
& & 0 \ar[d] & 0 \ar[d] & & \\
& & \mathcal{K} \ar@{=}[r] \ar[d] 
& \mu_{X}^{*}\mathcal{I}(-E \times X) \ar[d] & & \\
&0 \ar[r] &\mathcal{G} \ar[d] \ar[r] & \mu_{X}^{*}\mathcal{I} \ar[r] \ar[d] & \nu_{X^{*}}\mathcal{O}_{D}(2D) \ar@{=}[d] \ar[r] & 0 \\ 
& 0 \ar[r] & \mathcal{J} \ar[r] \ar[d] 
& \mu_{X}^{*}\mathcal{I}|_{E \times X} 
\ar[r] \ar[d] & \nu_{X}^{*}(\pi^{+})^*\mathcal{O}_{D}(2D) 
\ar[r] & 0 \\
& &0 & 0, & & 
}
\]
where we put 
\[
\mathcal{J}:=
\nu^{*}_{X}(\pi^{+}_{X})^{*}\left(
I_{D/C \times X} 
\otimes \mathcal{O}_{\pi^{+}}(1) 
\right). 
\]

Applying $(-) \otimes^{\dL} \mathcal{O}_{E \times X}$
to the left column in the above diagram, 
we get the following long exact sequence: 
\begin{align*} 
0 \to 
\mathcal{T}or^1\left(
\mathcal{O}_{E \times X}, 
\mathcal{G}
\right)
\to 
\mathcal{T}or^1\left(
\mathcal{O}_{E \times X}, 
\mathcal{J}
\right)
& \to \\
\mu_{X}^{*}\mathcal{I}(-E \times X)|_{E \times X} 
\to \mathcal{G}|_{E \times X} 
\to \mathcal{J}
&\to 0
\end{align*}

Note that we have 
$\mathcal{T}or^1\left(
\mathcal{O}_{E \times X}, 
\mathcal{\mathcal{K}}
\right)=0$ 
since $\mathcal{I}$ is flat over $\Hilb^2(X)$. 
Using the isomorphism 
$\mathcal{O}_{E \times X} 
\cong \left(
\mathcal{O}_{M \times X}(-E \times X) 
\to \mathcal{O}_{M \times X} 
\right)$
in $D^b(M \times X)$, 
we can easily show that 
$\mathcal{T}or^1\left(
\mathcal{O}_{E \times X}, 
\mathcal{G}
\right)=0$ and 
the above long exact sequence 
splits into 
\begin{align*}
&0 \to \mathcal{J} (-E \times X) 
\to \mu_{X}^{*}\mathcal{I}(-E \times X)|_{E \times X} 
\to \left(\nu_{X}^{*}(\pi^{+})^*\mathcal{O}_{D}(2D)\right) (-E \times X) 
\to 0 \\ 
& 0 \to \left(\nu_{X}^{*}(\pi^{+})^*\mathcal{O}_{D}(2D) \right) (-E \times X) 
\to \mathcal{G}|_{E \times X} 
\to \mathcal{J}
\to 0. 
\end{align*}

Hence $\mathcal{G}$ is flat over $M$. 
Furthermore, 
$\mathcal{G}|_{E \times X}$ 
parametrizes 
extensions 
\[
0 \to \mathcal{O}_{L_{y}}(-2) 
\to G 
\to I_{L_{y}} 
\to 0, 
\]
i.e. $1$-stable objects $G$. 
This defines a morphism 
$\sigma \colon M \to M^1(v)$. 
We claim that the image of $\sigma$ 
is contained in a connected component 
$\widetilde{M}^1(v)$. 
Indeed, we have 
$\sigma(M \setminus E) \subset \widetilde{M}^1(v)$ 
by our definition of $\widetilde{M}^1(v)$. 
Since $\sigma(M)$ is connected, it must be contained 
in $\widetilde{M}^1(v)$. 

(2) The assertion follows from 
\cite[Proposition A.2]{fri95}. 
\end{proof}

\begin{cor}
\label{flip}
The scheme $\widetilde{M}^1(v)$ 
is a smooth projective variety. 
Moreover, 
\begin{enumerate}
\item When $n=2$, 
$\Hilb^2(X) \dashrightarrow \widetilde{M}^1(v)$ is a flip. 
\item When $n=3$, 
$\Hilb^2(X) \dashrightarrow \widetilde{M}^1(v)$ is a flop. 
\item When $n \geq 4$, 
$\Hilb^2(X) \dashrightarrow \widetilde{M}^1(v)$ is an anti-flip. 
\end{enumerate}
\end{cor}

\begin{proof}
Note that Lemma \ref{elementary transform} shows that 
the diagram 
\[
\xymatrix{
& & M \ar[ld]^{\sigma} \ar[rd]_{\mu} & \\ 
& \widetilde{M}^1(v) & & \Hilb^2(X) 
}
\]
is a family version of the standard flip 
(cf. \cite[page 258]{huy06}). 
Since we already know the projectivity 
of $\widetilde{M}^1(v)$, it remains to check 
that $\widetilde{M}^1(v)$ is a smooth variety. 
The assertion then follows from 
the Fujiki-Nakano criterion 
(cf. \cite{fn71} and \cite[page 259]{huy06}). 
\end{proof}


\subsection{$\xi_{0}$ is an extremal contraction}

In this subsection, we describe the 
fiber of $\xi_{0}$ and show that 
it is the contraction of a $K$-negative extremal ray. 
Recall from Corollary \ref{fiber of xi1} (3) that 
we have 
\[
\widetilde{M}^1(v)=
\left(
\Hilb^2(X) \setminus \Hilb^2(D/C)
\right) 
\coprod 
\mathbb{P}\left(
\mathcal{E}_{-}^{\vee}
\right). 
\]

We start with the following lemma: 

\begin{lem}
Let $G \in \mathbb{P}(\mathcal{E}_{-}^{\vee}) \subset \widetilde{M}^1(v)$. 
Then $G$ is not $0$-stable. 
\end{lem}

\begin{proof}
We have an exact sequence 
\[
0 \to \mathcal{O}_{L_{y}}(-2) \to G \to I_{L_{y}} \to 0
\]
for some $y \in C$. 
Applying $\Hom(-, \mathcal{O}_{L_{y}}(-1))$, 
we get 
\begin{align*}
&\Hom\left(
I_{L_{y}}, \mathcal{O}_{L_{y}}(-1)
\right) 
\to \Hom\left(
G, \mathcal{O}_{L_{y}}(-1)
\right) 
\to \Hom\left(
\mathcal{O}_{L_{y}}(-2), \mathcal{O}_{L_{y}}(-1)
\right) \\ 
\to 
&\Ext^1\left(
I_{L_{y}}, \mathcal{O}_{L_{y}}(-1)
\right)
\end{align*}

Since $I_{L_{y}} \in \Per(X/Y)$, 
$\Hom\left(
I_{L_{y}}, \mathcal{O}_{L_{y}}(-1)
\right)=0$. 
Moreover, as in Lemma \ref{dim of fiber}, 
we can show that 
$\ext^1\left(
I_{L_{y}}, \mathcal{O}_{L_{y}}(-1)
\right)=1$. 
Hence 
$\hom\left(
G, \mathcal{O}_{L_{y}}(-1)
\right)=1 \neq 0$. 
This shows that $G \notin \Per(X/Y)$. 
\end{proof}

By the above lemma and 
Lemma \ref{stability of ideal sheaf}, 
the exceptional locus of $\xi_{0}$ is 
\[
\Exc(\xi_{0})=
\left(
\left\{
I_{Z} \in \Hilb^2(X) : 
Z \cap D \neq \emptyset 
\right\}
\setminus 
\Hilb^2(D/C)
\right) 
\coprod 
\mathbb{P}\left(
\mathcal{E}_{-}^{\vee}
\right). 
\]

To study the geometry of $\xi_{0}$, 
we introduce the following filtration 
$B_{1} \subset B_{2} \subset B_{3}=B \subset \Hilb^2(Y)$ 
of $\Hilb^2(Y)$: 
\begin{itemize}
\item $B_{1}:=\left\{
I_{W} \in \Hilb^2(C) : 
\Supp(I_{W})=\pt
\right\}$, 

\item $B_{2}:=\Hilb^2(C)$, 

\item $B_{3}:=\left\{
I_{W} \in \Hilb^2(Y) : 
W \cap C \neq \emptyset
\right\}$. 
\end{itemize}

\begin{lem}
\label{determine fiber}
\begin{enumerate}
\item We have $\xi_{0}(\Exc(\xi_{0}))=B$. 

\item For $I_{W} \in B \setminus B_{2}$, 
the fibre of $\xi_{0}$ is 
$\xi_{0}^{-1}(I_{W})=\mathbb{P}^1$. 

\item For $I_{W} \in B_{2} \setminus B_{1}$, 
the fibre of $\xi_{0}$ is 
$\xi_{0}^{-1}(I_{W})=\mathbb{P}^1 \times \mathbb{P}^1$. 

\item For $I_{W} \in B_{1}$, 
the fibre of $\xi_{0}$ is 
the weighted projective plane 
$\xi_{0}^{-1}(I_{W})=\mathbb{P}(1, 1, 2)$. 

\end{enumerate}

Here, we consider the image and the fibers of $\xi_{0}$ 
with its reduced scheme structures. 
\end{lem}

\begin{proof}
(1) First let 
$I_{Z} \in \Exc(\xi_{0}) \cap \Hilb^2(X)$. 
Then pushing forward the exact sequence 
\[
0 \to I_{Z} \to \mathcal{O}_{X} \to \mathcal{O}_{Z} \to 0, 
\]
we get 
\[
0 \to f_{*}I_{Z} \to \mathcal{O}_{Y} \to f_{*}\mathcal{O}_{Z} \to 0. 
\]

Hence $f_{*}I_{Z}$ is an ideal sheaf of 
length $2$ closed subscheme $W$ of Y 
with $\Supp W=f(Z)$. 
Since $Z \cap D \neq \emptyset$, 
$W \cap C \neq \emptyset$, i.e. 
$f_{*}I_{Z} \in B$. 

Next take $G \in \mathbb{P}(\mathcal{E}_{-}^{\vee})$. 
We can easily check that 
the following natural map 
$\alpha \colon 
\Ext^1\left(
I_{L_{y}}, \mathcal{O}_{L_{y}}(-2)
\right) 
\to 
\Ext^1\left(
\mathcal{O}_{y}, \mathcal{O}_{y}
\right)$ 
determines the class of $f_{*}G \in \Hilb^2(Y)$: 
Consider the short exact sequence 
\[
0 \to \mathcal{O}_{L_{y}}(-1) \to f^{*}I_{y} \to I_{L_{y}} \to 0. 
\]
Applying $\Hom\left(
-, \mathcal{O}_{L_{y}}(-2)
\right)$, 
we get 
\begin{align*}
\alpha \colon 
\Ext^1\left(
I_{L_{y}}, \mathcal{O}_{L_{y}}(-2) 
\right) 
&\to 
\Ext^1\left(
f^{*}I_{y}, \mathcal{O}_{L_{y}}(-2) 
\right) \\
&\cong 
\Hom\left(
I_{y}, \dR f_{*}\mathcal{O}_{L_{y}}(-2)[1]
\right) \\
&\cong 
\Hom\left(
I_{y}, \mathcal{O}_{y}
\right) \\ 
&\cong 
\Ext^1\left(
\mathcal{O}_{y}, \mathcal{O}_{y}
\right). 
\end{align*}

Hence $f_{*}G \in B_{1}$. 
Furthermore, we have 
$\Ext^1\left(
\mathcal{O}_{L_{y}}(-1), \mathcal{O}_{L_{y}}(-2)
\right)=0$ 
as in Lemma \ref{dim of fiber} 
and hence $\alpha$ is injective. 
Since we know 
$\ext^1\left(
I_{L_{y}}, \mathcal{O}_{L_{y}}(-2)
\right) 
= 
\ext^1\left(
\mathcal{O}_{y}, \mathcal{O}_{y}
\right)=n$, 
$\alpha$ is bijective. 
We conclude that 
\[
\xi_{0}|_{\mathbb{P}\Ext^1\left(
I_{L_{y}}, \mathcal{O}_{L_{y}}(-2)
\right)^{\vee}} 
\colon 
\mathbb{P}\Ext^1\left(
I_{L_{y}}, \mathcal{O}_{L_{y}}(-2)
\right)^{\vee} 
\cong 
\mathbb{P}\Ext^1\left(
\mathcal{O}_{y}, \mathcal{O}_{y}
\right)^{\vee}
\]
is an isomorphism. 

(2) Take an element 
$I_{W} \in B \setminus B_{2}$. 
First assume that 
$W=\{
a, b
\}$, 
$a \in C, b \notin C$. 
Then from the argument of (1), 
we have 
\[
\xi_{0}^{-1}(I_{W})
=\left\{
I_{p, q} \in \Hilb^2(X) : 
f(p)=a, f(q)=b
\right\}
\cong L_{a}. 
\]

Next assume that $\Supp W=\{y\}$, 
but scheme-theoretically $W \nsubseteq C$. 
Let $x \in D$ with $f(x)=y$. 
Then the following commutative diagram 
of the tangent maps determines 
the morphism 
\[
\xi_{0} \colon 
\mathbb{P}\Ext^1_{X}\left(
\mathcal{O}_{x}, \mathcal{O}_{x}
\right)^{\vee} 
\setminus 
\mathbb{P}\Ext^1_{L_{y}}\left(
\mathcal{O}_{x}, \mathcal{O}_{x}
\right)^{\vee} 
\to 
\mathbb{P}\Ext^1_{Y}\left(
\mathcal{O}_{y}, \mathcal{O}_{y}
\right)^{\vee}. 
\] 
\begin{equation}
\label{tangent map}
\xymatrix{
& &0 \ar[d] & & & \\
& &T_{x}L_{y} \ar[d] & &0 \ar[d] & \\
&0 \ar[r] &T_{x}D \ar[r] \ar[d]^{\eta_{x}} & T_{x}X \ar[r] \ar[d]^{\phi_{x}} &N_{D/X}(x) \ar[r] \ar[d]^{\psi_{x}} &0 \\
&0 \ar[r] ]&T_{y}C \ar[r] \ar[d] &T_{y}Y \ar[r] &N_{C/Y}(y) \ar[r] &0 \\
& &0, & & & 
}
\end{equation}
where maps $\eta_{x}, \phi_{x}, \psi_{x}$ 
are the tangent maps. 
Note that the point 
$
[\psi_{x}\left(N_{D/X}(x)\right)^{\vee}]
\in 
\mathbb{P}(N_{C/Y}(y)^{\vee}) 
=L_{y}
$
is nothing but $x \in L_{y}$. 
Now take the $1$-dimensional subspace 
$
\mathbb{C} \cdot (\alpha, \beta) 
\subset 
T_{y}Y=T_{y}C \oplus N_{C/Y}(y)
$
which corresponds to $I_{W}$. 
Since we assume that 
$W \nsubseteq C$, 
we have 
$0 \neq \beta \in N_{C/Y}(y)$. 
Let 
$x:=[\left(
\mathbb{C} \cdot \beta^{\vee}
\right)
] \in L_{y}$. 
Then we have 
$\phi_{x}^{-1}\left(
\mathbb{C} \cdot (\alpha, \beta) 
\right)
=T_{x}L_{y} \oplus N_{D/X}(x)
$. 
Moreover, recall from (1) that 
we have the unique element 
$G_{W} \in \mathbb{P}\Ext^1\left(
I_{L_{y}}, \mathcal{O}_{L_{y}}(-2)
\right)^{\vee}$ 
such that $f_{*}G_{W}=I_{W}$. 
Hence we conclude that 
\begin{align*}
\xi_{0}^{-1}(I_{W})
&=\left(
\mathbb{P}\left(
T_{x}L_{y} \oplus N_{D/X}(x) 
\right)^{\vee}
\setminus 
\mathbb{P}\left(
T_{x}L_{y}
\right)^{\vee}
\right) 
\coprod 
\{
G_{W}
\} \\ 
&= 
\mathbb{A}^1 
\coprod 
\pt. 
\end{align*}

Since both 
$\widetilde{M}^1(v)$ 
and $\Hilb^2(Y)$ 
are smooth, 
$H^1(\mathcal{O}_{\xi_{0}^{-1}(I_{W})})=0$. 
Hence we must have 
$\xi_{0}^{-1}(I_{W})\cong \mathbb{P}^1$. 

(3) Take an element 
$I_{W} \in B_{2} \setminus B_{1}$. 
Then by definition, 
$W=\{a, b\}$ with 
$a, b \in C$, 
$a \neq b$. 
Hence we have 
\[
\xi_{0}^{-1}(I_{W})
=\left\{
I_{p, q} : 
f(p)=a, f(q)=b
\right\}
=L_{a} \times L_{b}
\cong \mathbb{P}^1 \times \mathbb{P}^1. 
\]

(4) Let $I_{W} \in B_{1}$, 
$\Supp W=\{y\} \subset C$. 
Take the subspace 
$\mathbb{C} \cdot \alpha 
\subset T_{y}Y$ 
corresponding to 
$I_{W}$. 
Since $W \subset C$ as scheme, 
we have 
$\mathbb{C} \cdot \alpha \subset T_{y}C$. 
For each $x \in L_{y}$, 
we have 
$\phi_{x}^{-1}(\mathbb{C} \cdot \alpha)
=\mathbb{C} \cdot \alpha \oplus T_{x}L_{y}$. 
If we change $x \in L_{y}$, 
the vector space $T_{x}L_{y}$ changes, but 
the subspace $\mathbb{C} \cdot \alpha$ 
does not change. 
As before, we also have 
the unique element 
$G_{W} \in \mathbb{P}\Ext^1\left(
I_{L_{y}}, \mathcal{O}_{L_{y}}(-2)
\right)^{\vee}$ 
such that 
$f_{*}G_{W}=I_{W}$. 
Hence we conclude that 
\begin{align*}
\xi_{0}^{-1}(I_{W})
&=
\left(
\mathbb{P}\left(
\left(
\mathcal{O}_{L_{y}} 
\oplus 
\mathcal{T}_{L_{y}}
\right)^{\vee} 
\right) 
\setminus 
\mathbb{P}\left(
\mathcal{T}_{L_{y}}^{\vee}
\right)
\right)
\coprod 
\{
G_{W}
\} \\
&\cong 
\left(
\mathbb{P}\left(
\mathcal{O}_{L_{y}} \oplus \mathcal{O}_{L_{y}}(-2) 
\right)
\setminus 
\mathbb{P}\left(
\mathcal{O}_{L_{y}}(-2)
\right)
\right) 
\coprod 
\{
G_{W}
\}. 
\end{align*}

By this description, 
we can see that 
$\xi_{0}^{-1}(I_{W})$ 
is the proper transform of 
$S:=\mathbb{P}\left(
\mathcal{O}_{L_{y}} \oplus \mathcal{O}_{L_{y}}(-2) 
\right)$ 
via the birational map 
$\widetilde{M}^1(v) \dashrightarrow \Hilb^2(X)$. 
To determine the scheme structure of 
$\xi_{0}^{-1}(I_{W})$, 
recall the diagram 
\[
\xymatrix{
& & M \ar[ld]^{\sigma} \ar[rd]_{\mu} & \\ 
& \widetilde{M}^1(v) & & \Hilb^2(X). 
}
\]

Since 
$S \cap \Hilb^2(D/C)
=\mathbb{P}\left(
\mathcal{O}_{L_{y}}(-2)
\right)=:b$, 
the proper transform 
of $S$ by $\mu$ is 
$\mu_{*}^{-1}(S) \cong S$. 
Then the morphism 
$\sigma|_{S} \colon S \cong \mu_{*}^{-1}(S) 
\to \xi_{0}^{-1}(I_{W})$ is 
nothing but the contraction of 
a $(-2)$-curve $b \subset S$ 
and hence we get 
$\xi_{0}^{-1}(I_{W}) \cong \mathbb{P}(1, ,1, 2)$ 
as required. 
\end{proof}

\begin{lem}
\label{relative pic}
The relative Picard number 
$\rho\left(
\widetilde{M}^1(v)/\Hilb^2(Y)
\right)$
is one. 
\end{lem}

\begin{proof}
By Lemma \ref{determine fiber}, 
the only codimension $1$ irreducible 
component in $\Exc(\xi_{0})$ is 
the closure of 
$\xi_{0}^{-1}\left(
B \setminus B_{2}
\right)$. 
Hence it is enough to show that 
$B$ is irreducible. 
To see that, we may assume 
$Y=\mathbb{C}^n
=\Spec\mathbb{C}[x_{1}, \cdots, x_{n}]$, 
$C=(x_{1}=x_{2}=0)$. 
Then we can write down the local equations 
of $B \subset \Hilb^2(Y)$ by using the fact that 
$\Hilb^2(Y)=\Bl_{\Delta_{Y}}\Sym^2(Y)$. 
The Jacobian criterion then 
shows that 
$B$ is smooth. 
In particular, it is irreducible. 
\end{proof}

\begin{cor}
\label{extremal contraction}
The birational morphism $\xi_{0}$ 
is the contraction of a 
$K$-negative extremal ray. 
\end{cor}

\begin{proof}
By Lemma \ref{relative pic}, 
we know that $\xi_{0}$ is 
the contraction of an 
extremal ray $R$. 
Hence it is enough to compute 
the intersection number 
$K_{\widetilde{M}^1(v)}. f$ 
for one element $f \in R$. 
Fix $y \in C, \  q \in X \setminus D$ 
and put 
$f:=\left\{
I_{p, q} : p \in L_{y}
\right\} \in R$. 
Since $f$ does not intersects with 
the exceptional divisor of 
the Hilbert-Chow morphism 
$\Hilb^2(X) \to \Sym^2(X)$, 
we have 
\[
K_{\widetilde{M}^1(v)}. f 
=K_{\Hilb^2(X)}. f 
=K_{X \times X}. 
\left(
L_{y} \times \{q\}
\right)
=-1. 
\]
\end{proof}

\begin{rmk}
\label{not grassmann}
When $n=2$, 
the locus 
$\Hilb^2(C) = \emptyset$. 
Hence all the fibers of $\xi_{0}$ 
are $\mathbb{P}^1$. 
In general, Nakajima and Yoshioka 
shows that every fibre of 
the zig-zag diagram (\ref{zig-zag}) 
is the Grassmann variety 
(see Theorem \ref{ny}). 

On the other hand, for $n \geq 3$, 
we have shown that 
$\mathbb{P}^1 \times \mathbb{P}^1$ 
and $\mathbb{P}(1, 1, 2)$ appear 
as the fibers of $\xi_{0}$. 
Of course, they are not Grassmann variety. 
Furthermore, 
$\mathbb{P}(1,1,2)$ is even singular. 
This shows that for $n \geq 3$, 
more complicated fibers appear in 
the zig-zag diagram. 
\end{rmk}

\begin{ACK}
I would like to thank my supervisor 
Professor Yukinobu Toda, who 
gave me various advices and comments. 
I would also like to thank 
Takeru Fukuoka, Yousuke Matsuzawa, 
Masaru Nagaoka, and Genki Ouchi 
for fruitful discussions, 
and Wahei Hara, Yuki Hirano
for reading the draft version of this article 
and giving useful comments. 

This work was supported by the program 
for Leading Graduate Schools, 
the Ministry of Education,Culture,Sports,Science and Technology, Japan;  
and 
Grant-in-Aid for Japan Society for the Promotion of Science Research Fellow [17J00664].
\end{ACK}

\end{document}